\documentclass{amsart}

\PassOptionsToPackage{ngerman,main=UKenglish}{babel}
\PassOptionsToPackage{english}{selnolig}
\input{preamble.tex}

\hypersetup{%
    pdfauthor={Jannis Weis},%
    pdftitle={Quasi-isometry Invariance of discrete Higher Filling functions},%
    linktoc=page,
    hidelinks,
    linktoc=all,
}

\title[
    Quasi-Isometry Invariance of discrete Filling Functions%
]{%
    Quasi-Isometry Invariance of \mbox{discrete Higher Filling Functions}%
}%

\author{Jannis Weis}
\address{\foreignlanguage{ngerman}{Fakultät für Mathematik, Karlsruher Institut für Technologie}, 76131 Karlsruhe, Germany}
\email{jannis.weis@kit.edu}

\date{\today}


\begin{document}

\begin{abstract}
    We prove that homological filling functions over a ring \(R\) equipped with the discrete norm are quasi-isometry invariants for
    all groups of type~$\FP_n$.
    This confirms a conjecture of Bader--Kropholler--Vankov
    in the case of discrete norms.
    The proof uses a technique of equipping free chain complexes with a geometric structure, allowing
    for analogues of cellular constructions in the purely algebraic setting.
    As a further application we prove quasi-isometry invariance for a weighted version of integral and discrete filling functions originally introduced in the study of the rapid decay property.
\end{abstract}

\maketitle


\section{Introduction}

For a finitely presented group $\Gamma$, the classical \emph{Dehn function} $\delta_\Gamma(l)$
quantifies the difficulty of solving the word problem by measuring the maximal filling area
of a null-homotopic word of length at most \(l\). From a computational viewpoint, $\delta_\Gamma$ reflects the
complexity of a brute-force approach to the word problem in $\Gamma$.
This extends naturally to higher dimensions: replacing loops by \(n\)\=/spheres
yields the \emph{homotopical} higher Dehn functions $\delta_\Gamma^n$. It is well known that
these functions are quasi-isometry invariants for groups of type $\F_n$, that is, groups admitting a classifying space with finite \(n\)\=/skeleton~\cite{alonso+wang+pride}.

Instead of filling spheres, one can also study the maximal filling volume of integral \(n\)\=/cycles,
giving the \emph{homological} higher filling functions $\filling{\bbZ,\Gamma}{n+1}$.
Homological filling functions have the advantage that they readily generalise to arbitrary normed
coefficient rings. This yields a large family of functions $\filling{R,\Gamma}{n}$ depending on both
the choice of coefficients and on the chosen norm \cites{li+manin2021,bader+kropholler+vankov25}.

The most common coefficient choices for studying finiteness properties and filling functions -- besides the integers
with the Euclidean norm -- are finite fields equipped with the discrete norm, where the size of a chain is the cardinality
of its support~\cites{kropholler,li+manin2021,kielak+kropholler2021}.
More generally, any ring can be given the discrete norm, and we denote the resulting \emph{discrete} filling function
by $\dfilling{R,\Gamma}{n}$.

A natural question is whether the homological filling functions $\filling{R,\Gamma}{n}$ are also quasi-isometry
invariants of the group. In the case of integral filling functions $\filling{\bbZ,\Gamma}{n}$
this is known to be true for groups of type $\mathrm{F}_n$ due to Fletcher~\cites{fletcher} and
Young~\cite{young}.
Recent work of Bader, Kropholler, and Vankov~\cite{bader+kropholler+vankov25} shows that
$\filling{R,\Gamma}{n}$ is a quasi-isometry invariant for groups of type~$\FH_n(R)$,
provided the functions take only finite values for both groups.
Moreover, they proved that if a group $\Gamma$ is of type $\FH_n(R)$, then $\dfilling{R,\Gamma}{n}$ is finite.
They conjectured that a similar statement should hold for all groups of type $\FP_n(R)$ and arbitrary coefficients.

\begin{conjecture*}[Bader--Kropholler--Vankov]
    Let $G\noic$ and $H\noic$ be quasi-isometric groups of type $\FP_n(R)$. Then $\filling{R,G}{n} \approx \filling{R,H}{n}$.
\end{conjecture*}

Note that it is a well-known open problem, whether $\FP_n(R)$ and $\FH_n(R)$ are equivalent for $n > 2$.

\subsection*{Discrete filling functions}
In this work, we prove the conjecture of Bader--Kropholler--Vankov for discrete filling functions, completing the picture
of quasi-isometry invariance for all standard coefficient choices.

\getkeytheorem{qiInvarianceDiscrete}

We deduce \cref{thm: QI-invariance discrete} by first proving finiteness of discrete
filling functions for groups of type~$\FP_n(R)$.

\getkeytheorem{finitenessDFV}

\Cref{thm: QI-invariance discrete} then follows from a generalisation of the criterion given by~\cite{bader+kropholler+vankov25}.

\getkeytheorem{qiInvarianceIfFinite}

As it is already known that $\filling{\bbZ,\Gamma}{n}$ is finite for all groups of type $\FP_n(\bbZ)$ with $n \geq 2$, this theorem also implies quasi isometry invariance for integral filling functions in the
non-finitely presented case.

\subsection*{Weighted filling functions}
The techniques used in the proof of \cref{thm: quasi-isometry invariance of filling} generalise to a
\emph{weighted} variation of homological filling functions.
Any free $R\Gamma$\=/module $L=\bigoplus_{b\in B}R\Gamma b$ carries an $\ell^1$\=/norm obtained by
viewing it as a free \(R\)\=/module with basis $\Gamma B$ and setting
$\norm{\sum a_{\gamma b}\gamma b}=\sum\abs{a_{\gamma b}}$. In the weighted variant, each
basis element $\gamma b$ is assigned the weight $1+\ell_\Gamma(\gamma)$, where
$\ell_\Gamma$ denotes the word length in~$\Gamma$. The resulting norm
$\norm{\placeholder}^\Gamma$ and the corresponding weighted filling functions
$\weightedfilling{R,\Gamma}{n}$ appear naturally in several situations.
They were first studied by Ogle~\cite{ogle} and Ji--Ramsey~\cite{ji+ramsey} in connection
with the Rapid Decay property.
They were also utilised by Bader--Sauer in their work on higher property ($\mathrm{T}$)~\cite{bader+sauer} to deduce isomorphisms in polynomial cohomology under the assumption that $\weightedfilling{\bbZ,\Gamma}{n}$
is polynomially bounded. Furthermore,
weighted filling functions are used in upcoming joint work with Roman Sauer to give a
characterization for polynomiality of filling functions using a homological algebra framework~\cite{sauer+weis}.

Our techniques for non-weighted filling functions extend to this setting, yielding
weighted versions of \cref{thm: quasi-isometry invariance of filling,thm: QI-invariance discrete}.

\getkeytheorem{qiInvarianceDiscreteWeighted}

\getkeytheorem{qiInvarianceIntegralWeighted}

\subsection*{Proof strategy}
The key step in the proof of \cref{thm: dfilling finite} is a translation of geometric
arguments on cellular chain complexes into an algebraic framework for free
$R\Gamma$\=/chain complexes. This approach, developed in \cref{sec: cellular arguments},
replaces cellular subcomplexes with finitely generated \(R\)\=/subcomplexes and reproduces the
relevant geometric constructions purely algebraically. Conceptually, under mild hypotheses
(which $\FP_n(R)$\=/resolutions can be arranged to satisfy), free $R\Gamma$\=/chain complexes
behave much like the cellular chain complexes of genuine spaces.

The resulting techniques
appear interesting in their own right and may admit further applications.
Indeed, an easy application of this framework is \cref{prop: finite fibers implies finite filling function},
which provides another proof of the finiteness theorem of~\cite{fleming+martinez-pedroza} for the integral filling functions.
We also present a new proof of the quasi-isometry invariance of cohomology with coefficients
in the group ring for groups of type $\FP(R)$.

Once we developed this algebraic framework,
the proof of \cref{thm: dfilling finite} is a generalisation of the corresponding
geometric proof in~\cite{bader+kropholler+vankov25} to this setting.

\subsection*{Structure}
The paper is organised as follows. In \cref{sec: filling functions} we recall the
definition of filling functions over normed rings and establish basic properties needed
throughout. \Cref{sec: cellular arguments} develops the algebraic counterpart of
geometric arguments on $\Gamma$\=/complexes, which is then applied in
\cref{sec: finiteness} to prove the finiteness of discrete filling functions.
\Cref{sec: QI invariance} gives the proof of \cref{thm: quasi-isometry invariance of filling}
and its weighted counterpart.
Finally, \cref{sec: cohomological dimension} contains the proof
of the quasi-isometry invariance of cohomology using the techniques
from \cref{sec: cellular arguments, sec: QI invariance}.

\subsection*{Acknowledgements}
I would like to thank my advisors Claudio Llosa Isenrich and Roman Sauer for their valuable comments and suggestions during the creation of this work and Kevin Li for suggesting the content of \cref{sec: cohomological dimension} as a possible application.
This research was supported by the \foreignlanguage{ngerman}{Deutsche Forschungsgemeinschaft} (DFG, German Research Foundation) through
the Research Training Group DFG 281869850 (RTG 2229).

\section{Homological filling functions over normed rings}\label{sec: filling functions}
Following~\cite{bader+kropholler+vankov25} and~\cite{kielak+kropholler2021}, we give a short introduction to
homological filling functions over arbitrary normed rings.

Throughout, let $\Gamma$ be a group and \(R\) a ring with unit $1 \neq 0$.
All modules considered are left modules unless indicated otherwise.

\subsection{Normed rings}

\begin{defn}
    A \emph{norm} on a ring \(R\) is a function $\abs{\placeholder} \colon R \to \bbR_{\geq 0}$ such that
    \begin{condenum}
        \item $\abs{x} \geq 0$ and $\abs{x} = 0$ if and only if $x = 0$;
        \item $\abs{x + y} \leq \abs{x} + \abs{y}$;
        \item $\abs{xy} \leq \abs{x}\abs{y}$.
    \end{condenum}
    A ring together with a choice of norm is called a \emph{normed ring}.
    Two norms $\abs{\placeholder}_1, \abs{\placeholder}_2$ are said to be \emph{equivalent} if there exists a constant $C > 0$ such that
    \[
        \frac{1}{C} \cdot \abs{x}_2 \leq \abs{x}_1 \leq C \cdot \abs{x}_2
    \]
    for all $x \in R$.
\end{defn}

\begin{example}%
    \label{exa: norms}\AvoidPageBreak\leavevmode
    \begin{enumerate}
        \item If \(R\) is a subring of $\bbC$, the restriction of the absolute value turns \(R\) into a normed ring.
        \item Any ring can be made into a normed ring by equipping it with the \emph{discrete norm}, given by
              \[
                  \dabs{x} =
                  \begin{cases}
                      0 & \text{if } x = 0\mpct{;}    \\
                      1 & \text{if } x \neq 0\mpct{.}
                  \end{cases}
              \]
        \item Let $\Gamma$ be a finitely generated group with word metric~$d_\Gamma$
              and associated word length $\ell_\Gamma(\gamma)=d_\Gamma(\gamma,e)$.
              Let \(R\) be a normed ring. The associated \emph{weighted norm} on the
              group ring $R\Gamma$ is defined as
              \begin{align*}
                  \norm{\sum_{\gamma \in \Gamma} a_\gamma \gamma}^\Gamma \defq \sum_{\gamma \in \Gamma} \abs{a_\gamma} (1 + \ell_\Gamma(\gamma))\mpct{.}
              \end{align*}
              It is a straightforward computation that this is a norm on $R\Gamma$.
    \end{enumerate}
\end{example}

For any normed ring \(R\), there is an induced \emph{$\ell^1$\=/norm} on any free \(R\)\=/module with basis \(B\) given by
\[
    \norm{\sum_{b \in B} \alpha_b b} = \sum_{b \in B} \abs{\alpha_b}\mpct{.}
\]
We can consider the group ring $R\Gamma$ as a free \(R\)\=/module with basis $\Gamma$ and equip it with the $\ell^1$\=/norm.

Let $\Gamma$ be finitely generated. On free $R\Gamma$\=/modules \(L\) we consider two separate
norms. The $\ell^1$\=/norm $\norm{\cdot}$ induced by viewing \(L\) as a free \(R\)\=/module
and the weighted norm $\norm{\cdot}^\Gamma$ coming from the weighted norm on $R\Gamma$.

If \(R\) is equipped with the discrete norm, we write $\dnorm{\cdot}$ and $\dnorm{\cdot}^\Gamma$ for the corresponding induced norms.

The following two properties of a ring are sometimes convenient to assume.

\begin{defn}%
    \label{def: norm properties}
    A norm on a ring is called \emph{symmetric} if $\abs{{-x}} = \abs{x}$ for all $x \in R$.
    A norm is called \emph{$\epsilon$\=/separated} for $\epsilon > 0$ if $\abs{x} \geq \epsilon$ for all $x \in R \setminus \{0\}$.
\end{defn}

Examples of symmetric, $\epsilon$\=/separated norms include the absolute value norm on $\bbZ$
and the discrete norm on any ring. For both we can choose $\epsilon = 1$.
In \cref{sec: finiteness:non-finite}, we give an example showing that the assumption of an $\epsilon$\=/separated norm is essential for $\filling{R,\Gamma}{n}$ to take finite values for all groups.

\begin{lem}%
    \label{lem: symmetric norm}
    Every norm on a ring $R\noic$ is equivalent to a symmetric norm.
\end{lem}
\begin{proof}
    One easily checks that $\abs{r}' \defq \abs{r} + \abs{{-r}}$ defines a norm on \(R\). By definition, this norm
    is symmetric and $\abs{r} \leq \abs{r}' \leq \abs{r} + \abs{{-1}}\abs{r} = (1 + \abs{{-1}})\abs{r}$.
\end{proof}

\begin{lem}%
    \label{lem: 0-separated norm bounded by 1}
    Let $R\noic$ be an $\epsilon$\=/separated normed ring. Then up to equivalence of norms we may assume that \(R\) is \(1\)\=/separated.
\end{lem}
\begin{proof}
    Scaling a norm by a factor $C \geq 1$ defines a norm which is equivalent to the original one.
    Therefore, if $\epsilon < 1$ in \cref{def: norm properties}, we can scale the norm by $\frac{1}{\epsilon}$.
\end{proof}

The following \lcnamecref{lem: equivariant bounded} is proved in~\cite{bader+kropholler+vankov25}*{Lemma 2.3} for the
non-weighted $\ell^1$\=/norm on free modules. In the weighted case the proof is identical.

\begin{lem}%
    \label{lem: equivariant bounded}
    Let $\Gamma$ be a finitely generated group and $f \colon M \to N\noic$ be an $R\Gamma$\=/linear map between finitely generated
    free $R\Gamma$\=/modules. Then there exists a constant $C > 0$ such that $\norm{f(x)}^\Gamma \leq C \norm{x}^\Gamma$
    for all $x \in M$.
\end{lem}

\subsection{Filling functions}%
\label{sec: filling functions def}

\begin{defn}
    Let $n \geq 0$. A group $\Gamma$ is of type~$\FP_n(R)$ if there exists a projective $R\Gamma$\=/resolution
    $P_\ast \to R$ of the trivial $R\Gamma$\=/module \(R\), where $P_i$ is finitely generated for $i \leq n$.
    By~\cite{brown}*{VIII 4.3}, one may equivalently require the resolution to consist of free modules.
    If the resolution $P_\ast$ can be chosen to be finite, that is, $P_i = 0$ for $i > n$ for some \(n\), we say that
    $\Gamma$ is of type $\FP(R)$. If \(n\) is minimal with this property $\cd_R(\Gamma) \defq n$ is called the \emph{cohomological dimension}
    of $\Gamma$.
\end{defn}

\begin{rem}%
    \label{rem: cayley resolution}
    Let $\Gamma$ be of type~$\FP_n(R)$ with $n \geq 1$. Then there exists a free $R\Gamma$\=/resolution $L_\ast$ of the trivial
    $R\Gamma$\=/module \(R\) such that $L_i$ is finitely generated for $i \leq n$.
    We may assume that the modules $L_1$ and $L_0$ are the cellular chain modules of a Cayley graph of $\Gamma$.
    That is, if \(S\) is a finite generating set for $\Gamma$, then we may choose $L_\ast$ to look like the following chain complex
    \[
        \cdots \to L_2 \longrightarrow R\Gamma^{\size{S}} \xrightarrow{e_i \mapsto (s_i - 1)} R\Gamma \overset{\epsilon}{\longrightarrow} R\mpct{.}
    \]
\end{rem}

\begin{defn}[\cite{kielak+kropholler2021}*{Definition 2.1}]
    Let $n \geq 0$. We say that a projective $R\Gamma$\=/resolution $P_\ast \to R$ is \emph{\(n\)\=/admissible} if
    $P_n$ and $P_{n+1}$ are finitely generated free $R\Gamma$\=/modules equipped with a choice of free bases.
\end{defn}

The following definition is based on~\cite{bader+kropholler+vankov25}*{Definitions 2.9, 2.13}.

\begin{defn}%
    \label{def: filling function}
    Let \(R\) be a normed ring and $\Gamma$ be a group of type~$\FP_{n+1}(R)$. Let $P_\ast \to R$
    be an \(n\)\=/admissible resolution. The \emph{filling volume} of a cycle $c \in P_{n}$ is
    \[
        \fillvolume[R](c) \defq \;\;\inf_{\mathclap{\substack{b \in P_{n+1}\\ \partial b = c}}}\; \norm{b}\mpct{.}
    \]
    The \emph{weighted filling volume} of \(c\) is defined as
    \[
        \weightedfillvolume[R](c) \defq \;\;\inf_{\mathclap{\substack{b \in P_{n+1}\\ \partial b = c}}}\; \norm{b}^\Gamma\mpct{.}
    \]
    The \emph{\(n\)th homological filling function of $\Gamma$ with coefficients in \(R\)} is the function
    $\altfilling{R,\Gamma}{n}(l) \colon \bbR_{\geq 0} \to \bbR_{\geq 0} \cup \{\infty\}$ defined through
    \[
        \altfilling{R,\Gamma}{n}(l) \defq \;\sup_{\mathclap{\crampedsubstack{c\in P_{n}\\\partial c=0, \norm{c} \leq l}}}\; \fillvolume[R](c)\mpct{.}
    \]
    The weighted version is defined similarly by
    \[
        \weightedaltfilling{R,\Gamma}{n}(l) \defq \;\sup_{\mathclap{\crampedsubstack{c\in P_{n}\\\partial c=0, \norm{c}^\Gamma \leq l}}}\; \weightedfillvolume[R](c)\mpct{.}
    \]
\end{defn}

\begin{defn}
    Given functions $f, g \colon \bbR_{\ge 0} \to \bbR_{\ge 0}$ we write $f \preccurlyeq g$ if there are constants
    $C, D > 0$ such that $f(v) \leq C g(Cv + D) + Cv + D$ for every $v \geq 0$. We write $f \approx g$ if both
    $f \preccurlyeq g$ and $g \preccurlyeq f$ hold.
\end{defn}

The following \lcnamecref{lem: filling functions well-defined} from~\cite{bader+kropholler+vankov25} tells us that up to
$\approx$\=/equivalence $\altfilling{R,\Gamma}{n}$ and $\weightedaltfilling{R,\Gamma}{n}$ are well-defined
notions. There it is proven for non-weighted filling functions, however the proof
also applies if one replaces all norms by their weighted versions.

\begin{lem}[\cite{bader+kropholler+vankov25}*{Lemma 2.12}]%
    \label{lem: filling functions well-defined}
    Let $P_\ast\Noic$ and $P'_\ast\Noic$ be two \(n\)\=/admissible resolutions of the trivial $R\Gamma$\=/module \(R\). Let $\altfilling{}{n}$ and $\altfilling[\prime]{}{n}$ denote
    the corresponding (weighted) filling functions from \cref{def: filling function}, respectively. Then $\altfilling{}{n} \approx \altfilling[\prime]{}{n}$.
\end{lem}

\begin{convention}
    For a ring \(R\) equipped with the discrete norm, the notation $\dfilling{R,\Gamma}{n}$ is standard for $\altfilling{R,\Gamma}{n-1}$.
    In this case we also write $\weighteddfilling{R,\Gamma}{n}$ for the weighted version.
    Note that there is an index shift between the two notations.
    In the following we use the index convention matching $\dfilling{R,\Gamma}{n}$.
    To avoid confusion we write $\filling{R,\Gamma}{n}$ and $\weightedfilling{R,\Gamma}{n}$
    for $\altfilling{R,\Gamma}{n-1}$ and $\weightedaltfilling{R,\Gamma}{n-1}$ respectively.
\end{convention}

\begin{lem}\AvoidPageBreak%
    \label{lem: equivalent filling from eqeuivalent norms}
    Equivalent norms on $R\noic$ induce equivalent (weighted) filling functions.
\end{lem}
\begin{proof}
    We show this for the non-weighted version. The proof is identical in the weighted case.
    Let $\abs{\placeholder}$, $\abs{\placeholder}'$ be equivalent norms and $C > 0$ such that $C^{-1}\abs{x}' \leq \abs{x} \leq C\abs{x}'$ for all $x \in R$.
    We write $R'$ for \(R\) equipped with the norm $\abs{\placeholder}'$ and $\norm{\placeholder}'$ for the $\ell^1$\=/norm induced by $\abs{\placeholder}'$.
    Clearly $C^{-1} \cdot \norm{c}' \leq \norm{c} \leq C \cdot \norm{c}'$ and
    $C^{-1} \cdot \fillvolume[R'](c) \leq \fillvolume[R](c) \leq C \cdot \fillvolume[R'](c)$.
    Therefore,
    \begin{align*}
        \filling{R,\Gamma}{n}(l)
         &= \sup\{ \fillvolume[R](c) \mid \norm{c} \leq l\} \\
         &\leq \sup\{ C \cdot \fillvolume[R'](c) \mid \norm{c}' \leq C \cdot l\}
        = C \cdot \filling{R',\Gamma}{n}(C \cdot l)\mpct{,}
    \end{align*}
    hence $\filling{R,\Gamma}{n} \preccurlyeq \filling{R',\Gamma}{n}$ and by symmetry $\filling{R',\Gamma}{n} \approx \filling{R,\Gamma}{n}$.
\end{proof}

\section{Cellular arguments in free chain-complexes}\label{sec: cellular arguments}

The goal of this chapter is to develop tools for translating geometric arguments in cell complexes to the setting of free chain complexes.

In the following, let $\Gamma$ be a finitely generated group with a word metric~$d_\Gamma$ and associated word length
$\ell_\Gamma(\gamma)=d_\Gamma(\gamma,e)$.
Let $n \in \bbN \cup \{ \infty \}$ and let $L_\ast$ be a non-negative free $R\Gamma$\=/chain complex such that $L_i$
is finitely generated for $i \leq n$. For each $i \leq n$ let $B_i$ be a finite $R\Gamma$\=/basis of $L_i$.
Then $\Gamma B_i = \{ \gamma b \mid \gamma \in \Gamma,\, b \in B_i \}$ is an \(R\)\=/basis for $L_i$.

\subsection{A geometric structure on chain complexes}

Let $\alpha = \sum_{u \in \Gamma B_i} a_{u} u \in L_i$. The \emph{\(R\)\=/support} of $\alpha$ is the set
\[
    \supp_R(\alpha) \defq \{u \in \Gamma B_i \mid a_{u} \neq 0\}\mpct{.}
\]
Similarly, we define the \emph{$\Gamma$\=/support} of $\alpha$ as
\[
    \supp_\Gamma(\alpha) \defq \{\gamma \in \Gamma \mid a_{\gamma b} \neq 0 \text{ for some } b \in B_i\}.
\]

The functions $d_\Gamma$ and $\ell_\Gamma$ can be extended to
\begin{align*}
    d_\Gamma &\colon \coprod_{i = 0}^n \Gamma B_i \times \coprod_{i = 0}^n \Gamma B_i \to \bbN_0\mpct{,}
    \quad
    \ell_\Gamma \colon \Gamma B_i \to \bbN_0
\end{align*}
by setting $d_\Gamma(\gamma b, \gamma' b') \defq d_\Gamma(\gamma, \gamma')$ and
$\ell_\Gamma(\gamma b) \defq \ell_\Gamma(\gamma)$.
Note that $d_\Gamma$ is only a pseudo-metric on $\coprod_i \Gamma B_i$.

\begin{rem}%
    \label{rem: norm estimate from distance}
    Let $x = \gamma_x b \in \Gamma B_i$ and $y = \gamma_y b' \in \Gamma B_j$.
    Then
    \[
        \ell_\Gamma(y) = \ell_\Gamma(\gamma_y) \leq \ell_\Gamma(\gamma_x) + d_\Gamma(\gamma_x, \gamma_Y) = \ell_\Gamma(x) + d_\Gamma(x, y)\mpct{,}
    \]
    hence $\norm{y}^\Gamma = 1 + \ell_\Gamma(x) \leq d_\Gamma(x,y) = \norm{x}^\Gamma + d_\Gamma(x,y)$.
\end{rem}

\begin{defn}
    Let \(U\) be a subset of the \(R\)\=/basis $\Gamma B_i$. We define a non-oriented graph $\Graph(U)$
    with vertex set \(U\) and edges given by $\{u, v\}$ for $u \neq v \in U$ such that
    $\supp_R(\partial u) \cap \supp_R(\partial v) \neq \emptyset$.
    For an element $\alpha \in L_k$ we define $\Graph(\alpha)$  as $\Graph(\supp_R(\alpha))$.
    We say that \(U\), respectively $\alpha$, is \emph{connected} if $\Graph(U)$, respectively $\Graph(\alpha)$, is connected.
\end{defn}

\begin{rem}
    For $V \subset U$ the graph $\Graph(V)$ is the induced subgraph of $\Graph(U)$ on the vertices \(V\).
    Therefore, if \(U\) is not connected, there exists a decomposition $U = \bigsqcup_j U_j$ such that each $U_j$ is connected and of maximal size corresponding to a connected component of $\Graph(U)$.
    In particular, if $\alpha \in L_i$ is not connected, this decomposition of $\supp_R(\alpha)$ induces a decomposition $\alpha = \sum_{j=1}^m \alpha_j$
    such that $\norm{\alpha} = \sum_{j=1}^{m} \norm{\alpha_j}$. If $\alpha$ is a cycle, then each $\alpha_i$ is also a cycle as no cancellation
    can happen between different connected components.
\end{rem}

\begin{example}\AvoidPageBreak
    If $\Gamma = \quot{\bbZ}{k\bbZ} = \langle t \mid t^k \rangle$ a free $R\Gamma$\=/resolution of \(R\) is given by
    \[
        \cdots \xrightarrow{N \cdot} R\Gamma \xrightarrow{(1 - t) \cdot} R\Gamma \xrightarrow{N \cdot} R\Gamma \xrightarrow{(1 - t) \cdot} R\Gamma \overset{\epsilon}{\longrightarrow} R\mpct{,}
    \]
    where $N = \sum_{i = 0}^{k-1} t^i$. For each \(i\) a free \(R\)\=/basis of $R\Gamma$ is
    $\Gamma B_i = \{t^j \mid 0 \leq j \leq k - 1\}$.
    We have $N \cdot t^j = N$ and $(1-t) \cdot t^j = t^j - t^{j+1}$.
    Therefore, $\Graph(\Gamma B_n)$ is a cycle on \(k\) vertices if \(n\) is odd and a complete graph if \(n\) is even, except for
    $n = 0$, as $\Graph(\Gamma B_0)$ does not contain any edges (see \cref{fig: cyclic group incidence graphs}).
\end{example}

\begin{figure}
    \newcommand\radius{1}%
    \newcommand\rlab{1.3}%
    \newcommand\n{7}%
    \newcommand{\placeDots}{%
        \foreach \i in {1,...,\n} {
            \coordinate (P\i) at ({90 + 360/\n * -(\i - 1)}:\radius);
            \coordinate (L\i) at ({90 + 360/\n * -(\i - 1)}:\rlab);
            \fill (P\i) circle (2pt);
        }
    }%
    \newcommand{\drawLabels}{%
        \foreach \i in {1,...,\n} {
            \pgfmathtruncatemacro{\j}{\i - 1}
            \node[inner sep=0pt] at (L\i) {$t^{\mathrlap{\j}\kern2pt}$};
        }
    }%
    \newcommand{\drawCaption}[1]{%
        \node[below=15pt of current bounding box.south, anchor=center] {#1};
    }%
    \centering
    \begin{tikzpicture}
        \placeDots
        \drawLabels
        \drawCaption{$\Graph(\Gamma B_{0})$}
    \end{tikzpicture}
    \hspace{1cm}
    \begin{tikzpicture}
        \placeDots
        \draw[thick] (P1) \foreach \i in {2,...,\n} { -- (P\i) } -- cycle;
        \drawLabels
        \drawCaption{$\Graph(\Gamma B_{2i + 1})$}
    \end{tikzpicture}
    \hspace{1cm}
    \begin{tikzpicture}
        \placeDots
        \foreach \i in {1,...,\n} {
            \foreach \j in {\i,...,\n} {
                \draw[thick] (P\i) -- (P\j);
            }
        }
        \drawLabels
        \drawCaption{$\Graph(\Gamma B_{2i})$}
    \end{tikzpicture}
    \caption{%
        \label{fig: cyclic group incidence graphs}
        The graphs associated to the chain modules of the standard free $R\Gamma$\=/resolution of $\Gamma = \quot{\bbZ}{7\bbZ}$.
    }
\end{figure}
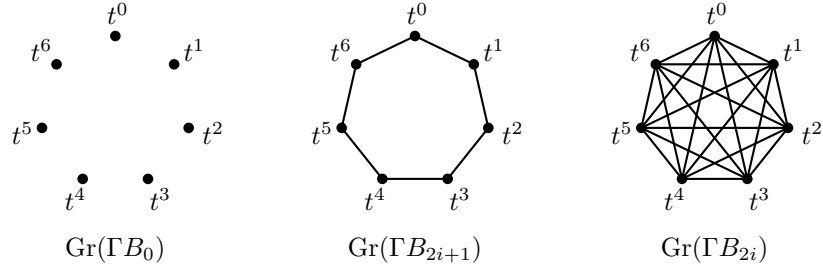

Assume that $L_\ast$ is the cellular chain complex of a free simplicial $\Gamma$\=/complex \(X\).
Then $\Gamma B_i$ corresponds bijectively to the set of \(i\)\=/cells of \(X\), and $\Graph(\Gamma B_i)$
is the graph whose edges connect cells sharing a common face.
We endow \(X\) with the metric induced by the Euclidean metric on each simplex, which in turn
defines a metric on the sets of simplices via the distance between their barycentres.
This allows us to bound the distance between neighbouring cells.
The next \lcnamecref{lem: norm estimate connected} provides an algebraic analogue of this phenomenon.

\begin{lem}%
    \label{lem: norm estimate connected}
    Let $0 \leq i \leq n$ and $U \subseteq \Gamma B_i$. There exists $K_i \geq 0$ such that
    \begin{refenum}
        \item if $x \in U\noic$ and $v \in \supp_R(\partial x)$,
              then $d_\Gamma(x, v) \leq K_i$ and $\norm{x}^\Gamma \leq \norm{v}^\Gamma + K_i$;
        \item if $x,y \in U\noic$ are connected by an edge in $\Graph(U)$, then
              $d_\Gamma(x, y) \leq 2K_i$ and $\norm{y}^\Gamma \leq \norm{x}^\Gamma + 2\cdot K_i$.
    \end{refenum}
\end{lem}
\begin{proof}
    In both cases we just show the estimate on the distance, as the statement for the norm follows from
    \cref{rem: norm estimate from distance}.
    Consider the set $A_i \defq \bigcup_{b^{(i)} \in B_i} \supp_\Gamma(\partial b^{(i)})$.
    If $A_i$ is empty then either $B_i$ is empty or $\partial_i = 0$. In either case both statements of the \lcnamecref{lem: norm estimate connected} are vacuously true, so we may set $K_i = 0$.
    From now on let us assume that $A_i$ is non-empty and define $K_i \defq \max\{\ell_\Gamma(h) \mid h \in A_i\} < \infty$.

    Let $x = \gamma_x b_x^{(i)} \in U$ and $v = \gamma_v b_v^{(i-1)} \in \supp_R(\partial x)$.
    Since $\partial x = \gamma_x \partial b_x^{(i)}$, we can write $\gamma_v = \gamma_x h$ for some $h \in A_i(b_x^{(i)})$.
    It follows that $d_\Gamma(x, v) = \ell_\Gamma(h) \leq K_i$.

    Now let $y = \gamma_y b_y^{(i)}$ be another element in \(U\) such that \(x\) and \(y\) are
    connected by an edge in $\Graph(U)$.
    Then there exist $h_x, h_y \in A_i$ and an $R\Gamma$\=/basis
    element $b^{(i-1)} \in B_{i-1}$ such that
    \[
        \gamma_x h_x b^{(i-1)} = \gamma_y h_y b^{(i-1)} \in \supp_R(\partial x) \cap \supp_R(\partial y)\mpct{.}
    \]
    Hence, $\gamma_y = \gamma_x h_x h_y^{-1}$ and therefore
    $d_\Gamma(x,y) \leq \ell_\Gamma(h_x) + \ell_\Gamma(h_y) \leq  2K_i$.
\end{proof}

\begin{cor}%
    \label{cor: estimate weighted to non-weighted}
    Let $0 \leq i \leq n$, $c \in L_i$ be connected, and let $x, y \in \supp_R(c)$.
    \begin{refenum}
        \item $d_\Gamma(x,y) \leq \size{\supp_R(c)} \cdot 2K_i$ and
              $\norm{y}^\Gamma \leq \norm{x}^\Gamma + \size{\supp_R(c)} \cdot 2K_i$;
    \end{refenum}
    Further assume that \(R\) is \(1\)\=/separated.
    \begin{refenum}[resume]
        \item $\norm{y}^\Gamma \leq \norm{x}^\Gamma + 2K_i \norm{c}$ and
              $\norm{c}^\Gamma \leq (\norm{x}^\Gamma + 2K_i \norm{c})\norm{c}$.
    \end{refenum}
\end{cor}

\begin{proof}
    The diameter of $\Graph(c)$ is at most $\size{\supp_R(c)}$, hence the first claim follows from \cref{lem: norm estimate connected}.
    If \(R\) is \(1\)\=/separated, $\norm{c} \geq \size{\supp_R(c)}$.
    Writing $c = \sum_{\gamma b \in \Gamma B_i} \alpha_{\gamma b} \gamma b$ we obtain
    \begin{align*}
        \norm{c}^\Gamma
        = \sum_{\gamma b} \abs{\alpha_{\gamma b}} \norm{\gamma b}^\Gamma
         &\leq \sum_{\gamma b} \abs{\alpha_{\gamma b}} (\norm{x}^\Gamma + \size{\supp_R(c)} 2K_i) \\
         &\leq (\norm{x}^\Gamma + 2K_i \norm{c}) \norm{c}\mpct{.}\qedhere
    \end{align*}
\end{proof}

\begin{defn}
    Let $\alpha \in L_i$. We say $\alpha$ is \emph{reduced} if there does
    not exist a decomposition $\alpha = \alpha_1 + \alpha_2$ with $\alpha_1 \neq 0$ such that
    $\partial \alpha_1 = 0$.
\end{defn}

\begin{lem}%
    \label{lem: reduced filling estimate}
    There exists a constant $C_i \geq 0$ such that for every reduced chain $\alpha \in L_i$
    and every $u \in \supp_R(\alpha)$ there is $v_u \in \supp_R(\partial \alpha)$ with
    $d_\Gamma(u, v_u) \leq C_i \size{\supp_R(\alpha)}$ and thus
    $\norm{u}^\Gamma \leq \norm{v_u}^\Gamma + C_i \dnorm{\alpha}$.
\end{lem}
\begin{proof}
    Decompose $\alpha$ into its connected components $\alpha = \alpha_1 + \dots + \alpha_k$.
    Because $\alpha$ is reduced, $\partial \alpha_p \neq 0$ for every $1 \leq p \leq n$. Hence, we can
    pick vertices $u_p$ of $\Graph(\alpha_p)$ such that
    $\supp_R(\partial u_p) \cap \supp_R(\partial \alpha_p) \neq \emptyset$.
    Let $v_p \in \supp_R(\partial u_p) \cap \supp_R(\partial \alpha_p)$ be such an element in the intersection.
    By \cref{lem: norm estimate connected}, $\norm{u_p}^\Gamma \leq \norm{v_p}^\Gamma + K_i$ and by \cref{cor: estimate weighted to non-weighted},
    \begin{align*}
        \norm{u}^\Gamma \leq \norm{u_p}^\Gamma + \size{\supp_R(\alpha_p)} \cdot 2K_i \leq \norm{v_p}^\Gamma + K_i + \size{\supp_R(\alpha)} \cdot 2K_i
    \end{align*}
    for each $u \in \supp_R(\alpha_p)$. A suitable constant is therefore given by $C_i \defq 3K_i$.
\end{proof}

\begin{lem}%
    \label{lem: locally finite graph}
    $\Graph(\Gamma B_i)$ is locally finite for all $0 \leq i \leq n$.
\end{lem}
\begin{proof}
    Because the differentials of $L_\ast$ are $\Gamma$\=/equivariant, $\Gamma$ acts on $\Graph(\Gamma B_i)$ by graph automorphisms.
    It therefore suffices to show that
    \begin{align*}
        B_{\Graph(\Gamma B_i)}(b, 1) = \{\gamma b' \in \Gamma B_i \mid \{\gamma b', b\} \text{ is an edge in } \Graph(\Gamma B_i)\}
    \end{align*}
    is finite for every $b \in B_i$.
    To this end let $\gamma b' \in \Gamma B_i$ be connected to \(b\) by an edge in $\Graph(\Gamma B_i)$.
    Then by \cref{lem: norm estimate connected} we have that
    \begin{align*}
        \ell_\Gamma(\gamma) = \ell_\Gamma(\gamma b') \leq \ell_\Gamma(b) + 2K_i = 2K_i
    \end{align*}
    and hence $\ell_\Gamma(\gamma) \leq 2K_i$. As $\Gamma$ is finitely generated and $B_i$ is finite,
    we conclude that $B_{\Graph(\Gamma B_i)}(b,1)$ is finite.
\end{proof}

For $i = 0$ the above statement is trivially satisfied, as $B_{-1}$ is empty and therefore $\Graph(\Gamma B_0)$ contains no edges at all.

\begin{lem}\AvoidPageBreak%
    \label{lem: finite neighbouring cells in above degree}
    Let $0 \leq i \leq n - 1$ and $u \in \Gamma B_i$. Then the set
    \begin{align*}
        S(u) \defq \{v \in \Gamma B_{i+1} \mid u \in \supp_R(\partial v)\}
    \end{align*}
    is finite.
\end{lem}
\begin{proof}
    If $S(u)$ is empty or consists of only one element, there is nothing to show.
    Otherwise, fix some $v_0 \in S(u)$. For any $w \in S(u)$ such that $w \neq v_0$, the graph $\Graph(\Gamma B_{i+1})$ has $\{v_0,w\}$ as an edge.
    Therefore, $S(u)$ is contained in $B_{\Graph(\Gamma B_{i+1})}(v_0, 1)$ which is finite by \cref{lem: locally finite graph}.
\end{proof}

\subsection{Non-equivariant subcomplexes}

We now study algebraic analogues of subcomplexes of a simplicial $\Gamma$\=/complex.
For $l > 0$ consider the collection
\begin{align*}
    \connSupp^i_l \defq \{ U \subset \Gamma B_i \mid \size{U} \leq l,\,U \text{ connected}\}\mpct{.}
\end{align*}
An element $U \in \connSupp^i_l$ can be thought of as giving a basis to a connected \(R\)\=/subspaces
of $L_i$ with \(R\)\=/dimension at most \(l\).
The group $\Gamma$ acts on $\connSupp^i_l$ by left multiplication, that is,
$\gamma U \defq \{\gamma u \mid u \in U\}$.

\begin{prop}\AvoidPageBreak%
    \label{prop: finite orbits of finite connected subcomplexes}
    Let $l \geq 0$ and $0 \leq i \leq n$. Then $\lquot{\connSupp^i_l}{\Gamma}$ is finite.
\end{prop}
\begin{proof}
    Let $U \in \connSupp^i_l$. If $U = \emptyset$, the orbit $\Gamma U$ just consists of \(U\). If $U \neq \emptyset$
    let $\gamma b \in U$. By considering $U' \defq \gamma^{-1}U$, we may assume that $\gamma = 1$ and $b \in U$.
    Because $\size{U} \leq l$, the diameter of $\Graph(U)$ is at most \(l\). Repeated application of
    \cref{lem: norm estimate connected} yields that for any $\gamma' b' \in U$ we have
    \begin{align*}
        \ell_\Gamma(\gamma') = \ell_\Gamma(\gamma'b')
         &\leq \ell_\Gamma(b) + 2K_i l = 2K_i l\mpct{.}
    \end{align*}
    Therefore, \(U\) is a subset of the finite set $\{\gamma \in \Gamma \mid l_\Gamma(\gamma) \leq 2K_i l\}B_i$.
\end{proof}

\begin{defn}
    A \emph{basis collection} is a collection $U_\ast$ of subsets $U_i \subseteq \Gamma B_i$ for $i \geq 0$.
    We say a basis collection $U_\ast$ is \emph{closed under differentials} if
    $\supp_R(\partial u) \subseteq U_{i-1}$ for all $u \in U_i$ and all $i > 0$.
    If $U_\ast$ is closed under differentials, the modules
    $L_i(U_\ast) \defq \langle U_i\rangle_R$ assemble to an \(R\)\=/subcomplex $L_\ast(U_\ast)$ of $L_\ast$.

    A basis collection $U_\ast$ is called \emph{at most \(k\)\=/dimensional} if $U_i = \emptyset$ for all $i > k$
    and \emph{\(k\)\=/dimensional} if additionally $U_k \neq \emptyset$.
    It is said to be of \emph{finite type} if $U_i$ is finite for all $i \geq 0$.
\end{defn}

\begin{rem}%
    \label{rem: subset as collection}
    We may understand a single subset $U \subseteq \Gamma B_k$ as a \(k\)\=/dimensional basis collection by choosing the empty subset
    of $\Gamma B_j$ for $j \neq k$. By abuse of notation, we also denote this collection by \(U\).
\end{rem}

Let $j \geq 0$ and $U_j \subseteq \Gamma B_j$ be some subset. We inductively define
\begin{align*}
    P^j_i(U_j) \defq
    \begin{cases}
        0                                                & \text{if } i > j\mpct{;} \\
        U_j                                              & \text{if } i = j\mpct{;} \\
        \bigcup_{u \in P^j_{i+1}(U)} \supp_R(\partial u) & \text{if } i < j\mpct{,}
    \end{cases}
\end{align*}
which is closed under differentials.
More generally, if we are given a collection $U_\ast$ of subsets $U_i \subseteq \Gamma B_i$ for all $i \geq 0$, we define
$P_i(U_\ast) \defq \bigcup_{j \geq 0} P^j_i(U_j)$. If $U_\ast$ is already closed under differentials then $P_\ast(U_\ast) = U_\ast$.

If $L_\ast$ is the cellular chain complex of some simplicial complex \(X\) a basis collection $U_\ast$ corresponds to a choice of subsets of \(i\)\=/cells
of \(X\) for each $i > 0$. Then $L_\ast(P_\ast(U_\ast))$ is the cellular chain complex of the subcomplex generated by these cells.

\subsection{Thickening of subcomplexes}
Next we describe a procedure which can be understood as an algebraic analogue of the following geometric construction.
Given a \(k\)\=/dimensional subcomplex \(Y\) of a simplicial complex \(X\), we define $N(Y)$ to be the subcomplex consisting of
\(Y\) together with all cells that have a boundary face in \(Y\).

Let $K_\ast$ be an \(R\)\=/subcomplex of $L_\ast$ equipped with a basis $U_i \subseteq \Gamma B_i$ for all \(i\).
First we consider the sets
\begin{align*}
    \tilde{N}_i(U_\ast) \defq
    \begin{cases}
        U_0                                   & \text{if } i = 0\mpct{;} \\
        U_i \cup \bigcup_{u \in U_{i-1}} S(u) & \text{if } i > 0\mpct{,}
    \end{cases}
\end{align*}
where $S(u)$ is as in \cref{lem: finite neighbouring cells in above degree}.
Note that if $U_\ast$ is at most \(k\)\=/dimensional, $\tilde{N}_i(U)$ is empty for $i > k + 1$.

On the level of simplicial complexes $\tilde{N}_i(U_\ast)$ corresponds to the collection of \(i\)\=/cells intersecting the
subcomplex \(Y\). In general this collection is not closed under differentials. For example let \(X\) be a simplicial tripod and
\(Y\) be the central vertex. Then $\tilde{N}_0 = Y_0$ and $\tilde{N}_1$ is the collection of all three edges.
However, these
sets do not form a simplicial complex because the outer vertices are missing (see \cref{fig: tripod}).
\begin{figure}
    \newcommand\radius{1}
    \begin{tikzpicture}
        \coordinate (T) at (0,0);

        \foreach \a/\n in {120/A,240/B,360/C} {
            \coordinate (\n) at ({\radius*sin(\a)}, {\radius*cos(\a)});
        }

        \draw[thick] (T) -- (A);
        \draw[thick] (T) -- (B);
        \draw[thick] (T) -- (C);

        \fill[red] (T) circle (2pt);

        \foreach \p in {A,B,C} {
            \fill[black] (\p) circle (2pt);
        }
    \end{tikzpicture}%
    \hspace{1cm}%
    \begin{tikzpicture}
        \coordinate (T) at (0,0);

        \foreach \a/\n in {120/A,240/B,360/C} {
            \coordinate (\n) at ({\radius*sin(\a)}, {\radius*cos(\a)});
        }

        \draw[red,thick] (T) -- (A);
        \draw[red,thick] (T) -- (B);
        \draw[red,thick] (T) -- (C);

        \fill[red] (T) circle (2pt);
        \foreach \p in {A,B,C} {
            \draw[black,fill=white] (\p) circle (2pt);
        }
    \end{tikzpicture}
    \caption{%
        \label{fig: tripod}
        On the left a tripod with subcomplex \(Y\) given by the central vertex. On the right the set $\tilde{N}_\ast(Y)$.
    }
\end{figure}
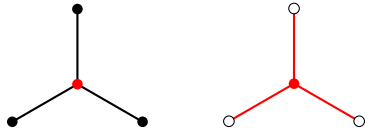
Applying $P_i(-)$ resolves this deficit.

For a chain complex $C_\ast$ we write $C_\ast^{(k)}$ for its \(k\)\=/skeleton, that is, $C_i^{(k)}$ is $C_i$ for $i\le k$ and \(0\) for $i>k$.
For a basis collection $U_\ast$, we write $U_{\ast \leq k}$ for the truncated collection with the empty set in degrees $i > k$.
Note that $L_\ast(U_{\leq k}) = L_\ast(U_\ast)^{(k)}$.

\begin{defn}
    Let $U_\ast$ be a basis collection. Let $k \in \bbN_0 \cup \{\infty\}$. For $j \geq 0$
    we inductively define the \emph{\(j\)\=/step \(k\)\=/thickening} of \(U\) as the basis collection
    \begin{align*}
        N^{k,j}_i(U_\ast) \defq
        \begin{cases}
            P_\ast(U_\ast)                                                                                      & \text{if } j = 0\mpct{;} \\
            P_i(\tilde{N}_{\ast \leq k}(P_\ast(U_\ast))) = \bigcup_{l = 0}^k P^l_i(\tilde{N}_l(P_\ast(U_\ast))) & \text{if } j = 1\mpct{;} \\
            N^{k,1}_\ast(N^{k,j-1}_\ast(U_\ast))                                                                & \text{if } j > 1\mpct{.}
        \end{cases}
    \end{align*}
    The associated \(R\)\=/subcomplexes are denoted by $N^{k,j}L_\ast(U)$.
\end{defn}

\begin{lem}%
    \label{lem: Criterion tilde N and P finite}
    Let $U_\ast\Noic$ be a basis collection. Then
    \begin{refenum}
        \item\label{lem: Criterion tilde N finite}
              $\tilde{N}_i(U_\ast)$ is finite if either
              \begin{condenum}
                  \item $i = 0$ and $U_0\Noic$ is finite, or
                  \item $1 \leq i \leq n$ and both $U_i\Noic$ and $U_{i-1}\Noic$ are finite;
              \end{condenum}
        \item\label{lem: Criterion P finite}
              $P_i(U_\ast)$ is finite if $U_j\Noic$ is finite for all $j \geq i$.
    \end{refenum}
\end{lem}
\begin{proof}
    The first implication follows directly from \cref{lem: finite neighbouring cells in above degree}.
    For the second statement note that $P_i(U_\ast)$ is finite if $P^j_i(U_j)$
    is finite for all $j \geq 0$. However, $P^j_i(U_j)$ is empty for $j < i$. If $j \geq i$ then
    by construction $P^j_i(U_j)$ is finite if and only if $U_j$ is finite.
\end{proof}

An immediate consequence is the following \lcnamecref{lem: thickening finite type}.
\begin{lem}%
    \label{lem: thickening finite type}
    Let $0 \leq k \leq n$ and $U_\ast\Noic$ be an at most \(k\)\=/dimensional basis collection of finite type.
    Then $N^{k,1}_i(U_\ast)$ is at most \(k\)\=/dimensional and of finite type.
\end{lem}

\begin{prop}%
    \label{prop: step thickening finite type}
    Let $0 \leq k \leq n$. Let $U_\ast\Noic$ be an at most \(k\)\=/dimensional basis collection of finite type.
    Then $N^{k,j}_\ast(U_\ast)$ is at most \(k\)\=/dimensional and of finite type for all $j \geq 0$.
\end{prop}
\begin{proof}
    By definition, $N^{k,j}_\ast(U_\ast)$ is always at most \(k\)\=/dimensional.
    Therefore, the claim follows from \cref{lem: thickening finite type} by induction on \(j\).
\end{proof}

\begin{thm}\enlargethispage{\baselineskip}%
    \label{thm: connectivity filtration}
    Let $1 \leq m \leq k \leq n$ and $U \in \connSupp^m_l$. Assume that
    \begin{condenum}
        \item\label{cond: U non-empty} $U\noic$ is non-empty;
        \item\label{cond: no-degenerate} $\partial b \neq 0$ for all $b \in B_i\Noic$ and $1 \leq i \leq k$;
        \item\label{cond: B1-connected} $\Graph(\Gamma B_1)$ is connected;
        \item\label{cond: B0-surjective} $\bigcup_{u \in \Gamma B_1} \supp_R(\partial u) = \Gamma B_0$.
    \end{condenum}
    Then $\colim_{j} N^{k,j} L_\ast(U)^{(k)} = L_\ast^{(k)}$.
\end{thm}
\begin{proof}
    We inductively show that $\bigcup_j N_i^{k,j}(U) = \Gamma B_i$ for all $i \leq k$.
    The case $i = 0$ follows directly from \lccref{cond: B0-surjective}.
    Note that $N^{k,0}_1(U) = P_1(U)$. From \lccref{cond: U non-empty, cond: no-degenerate} it follows that $P_1(U) \neq \emptyset$.
    Let $u \in P_1(U)$ and $v \in \Gamma B_1$. By \lccref{cond: B1-connected}, there is some path $p_{u,v}$ in $\Graph(\Gamma B_1)$
    from \(u\) to \(v\). Then $v \in N^{k,l}_1(U)$, where \(l\) is the length of the path $p_{u,v}$.

    Now let $1 < i \leq k$ and assume $\bigcup_j N_{i-1}^{k,j}(U) = \Gamma B_{i-1}$. For any $u \in \Gamma B_i$
    \lccref{cond: no-degenerate} guarantees that $\supp_R(\partial u)$ is non-empty.
    Choose \(j\) large enough such that $\supp_R(\partial u) \cap N_{i-1}^{k,j}(U) \neq \emptyset$, so $u \in N_{i}^{k,j+1}(U)$.
    Hence, $\bigcup_j N_{i}^{k,j}(U) = \Gamma B_{i}$.
\end{proof}

\section{Finiteness of filling functions}\label{sec: finiteness}

In this section we prove finiteness for various filling functions introduced in \cref{sec: filling functions def}.

\subsection{Integral filling functions}
The following \lcnamecref{prop: finite fibers implies finite filling function} recovers the result that
$\filling{\bbZ,\Gamma}{n}$ takes finite values~\cite{fleming+martinez-pedroza}.

\begin{prop}%
    \label{prop: finite fibers implies finite filling function}
    Let $n \geq 2$ and $\Gamma$ be a group of type~$\FP_n(R)$.
    Assume that for each $l \geq 0$ the set $R_{\leq l} \defq \{ r \in R \mid |r| \leq l \}$ is finite.
    Then $\filling{R,\Gamma}{n}$ takes finite values.
\end{prop}
\begin{proof}
    We consider the collection $\calD^{n-1}_l \defq \{ c \in L_{n-1} \mid c \text{ connected, } \norm{c} \leq l\}$ of connected
    $(n-1)$\=/cycles with norm at most \(l\). The group $\Gamma$ acts on $\calD^{n-1}_l$ by left multiplication,
    as $\norm{\gamma c} = \norm{c} \leq l$ for each $c \in \calD^{n-1}_l$ and $\gamma \in \Gamma$.
    By assumption, $R_{\leq 1}$ is finite, hence \(R\) is $\epsilon$\=/separate for some $\epsilon > 0$.
    Using \cref{lem: 0-separated norm bounded by 1} we may assume that $\epsilon = 1$,
    hence $\size{\supp_R(c)} \leq \norm{c} \leq l$.
    Let $U_1, \ldots, U_N$ be representatives of $\lquot{\connSupp^{n-1}_l}{\Gamma}$.
    If $c \in \calD^{n-1}_l$ is a cycle with $\supp_R(c) = U_j$, it is contained in the set
    \[
        R_{\leq l}U_j \defq \Bigl\{ \sum_{u \in U_j} \alpha_u u \mid \alpha_u \in R_{\leq l}\Bigr\}\mpct{,}
    \]
    which is finite as $R_{\leq l}$ and $U_j$ are finite. We conclude that $\lquot{\calD^{n-1}_l}{\Gamma}$ is finite.
    This implies that there is some $D(l) \geq 0$ such that $\fillvolume[R](c) \leq D(l)$ for every $c \in \calD^{n-1}_l$.
    In the general case a cycle \(c\) decomposes as a sum $c = c_1 + \cdots + c_s$ of connected cycles.
    Then $\norm{c} = \sum_i \norm{c_i}$ and each $c_i$ has a filling $b_i$ satisfying
    $\norm{b_i} \leq D(\norm{c_i})$. Therefore, $\sum_i b_i$ is a filling for \(c\) and we conclude
    \begin{equation*}
        \filling{R,\Gamma}{n}(l) \leq \max\Bigl\{ \sum_{i=1}^m D(l_i) \mid \sum_{i=1}^m l_i = l, l_i \in \bbN\Bigr\} < \infty\mpct{.}\qedhere
    \end{equation*}
\end{proof}

\begin{cor}%
    \label{cor: integral finite}
    $\filling{\bbZ,\Gamma}{n}$ takes finite values.
\end{cor}

\subsection{Discrete filling functions} To show that $\dfilling{R,\Gamma}{n}(l)$ is finite we
follow a similar approach as~\cite{bader+kropholler+vankov25}*{Proposition 2.21}.
There are only finitely many orbits for $\supp_R(c)$ of a connected cycle \(c\) with $\dnorm{c} \leq l$.
Therefore, we can use the thickening procedure from the previous section to obtain a subcomplex of finite \(R\)\=/dimension
containing a filling for all connected cycles. An upper bound on $\dnorm{b}$ for a filling \(b\) of \(c\) is given by
the dimension of this subcomplex in degree \(n\).
To apply \cref{thm: connectivity filtration} we first verify that there
exists a suitable resolution satisfying the assumptions of the \lcnamecref{thm: connectivity filtration}.

\begin{prop}%
    \label{prop: existence suitable resolution}
    Let $n \geq 1$ and $\Gamma$ be of type~$\FP_n(R)$. Then there exists a free $R\Gamma$\=/resolution $L_\ast \to R$ such that
    $L_i$ is finitely generated for $i \leq n$ and a choice of bases $B_i$ of $L_i$ such that
    \lccref{cond: no-degenerate, cond: B1-connected,cond: B0-surjective} from \cref{thm: connectivity filtration} are satisfied.
\end{prop}
\begin{proof}
    Let \(S\) be a finite generating set for $\Gamma$. As in \cref{rem: cayley resolution}, we start off our resolution with the cellular chain complex
    $L_i \defq C_i(\Cay(\Gamma, S); R)$ for $i = 0, 1$. Then $\Graph(\Gamma B_1)$ is the line graph of a Cayley graph of $\Gamma$%
    \footnote{The line graph
        of a graph \(G\) is the graph $L(G)$ with the edges of \(G\) as vertices and edges given by the intersection of edges in \(G\).}.
    In particular $L(G)$ is connected if \(G\) is connected.
    Therefore, $L_\ast$ satisfies \lccref{cond: B1-connected,cond: B0-surjective}.

    For \lccref{cond: no-degenerate} note that we can extend our partial resolution to a full resolution by iteratively choosing
    surjections $\partial_i \colon L_i \onto \ker(L_i \to L_{i-1})$ with $L_i$ of minimal $R\Gamma$\=/rank. If a basis element of
    $L_i$ were to be mapped to zero under $\partial_i$, this would violate minimality.
    The fact that we can choose $L_i$ to be finitely generated for $i \leq n$ follows from Schanuel's Lemma~\cite{brown}*{VIII 4.3}.
\end{proof}

\begin{thm}[store=finitenessDFV]%
    \label{thm: dfilling finite}
    Let $n \geq 2$ and $\Gamma$ be a group of type~$\FP_n(R)$. Then $\dfilling{R,\Gamma}{n}$ takes finite values.
\end{thm}
\begin{proof}
    We take $L_\ast$ to be the free resolution from \cref{prop: existence suitable resolution}.
    Let $l > 0$ and let $U_1,\ldots,U_m$ be a choice of representatives of the $\Gamma$\=/orbits in $\connSupp^{n-1}_l$.
    Because $H_{n-1}(L_\ast) = 0$, the inclusion $N^0L_\ast(U_i) \to L_\ast$ induces the trivial map in $(n-1)$th homology.
    Since homology commutes with direct limits and $H_{n-1}(N^0L_\ast(U_i))$ is a finitely generated \(R\)\=/module,
    it follows from \cref{thm: connectivity filtration} that there is an $a_i$ such that the inclusion
    $\iota_{a_i} \colon N^0L_\ast(U_i) \to N^{n,a_i}L_\ast(U_i)$ induces the trivial map
    \[H_{n-1}(\iota_{a_i}) \colon H_{n-1}(N^{n,0}L_\ast(U_i)) \to H_{n-1}(N^{n,a_i}L_\ast(U_i))\mpct{.}\]
    Define $M^i_\ast$ to be $N^{n,a_i}L_\ast(U_i)$.

    Let \(z\) be a connected $(n-1)$\=/cycle in $L_{n-1}$ with $\dnorm{z} \leq l$.
    Then there exists $1 \leq i \leq m$ and $\gamma \in \Gamma$ such that
    $\supp_R(z) = \gamma U_i$. As $\gamma^{-1}z$ and \(z\) have the same filling volume,
    we may assume that $\gamma = 1$.
    As $H_{n-1}(\iota_{a_i}) = 0$, there exists a filling $b \in M^i_n$. The norm of \(b\) is bounded by
    $\size{N^{n,a_i}_n(U_i)} \qdef D_i(l)$. By \cref{prop: step thickening finite type}, $D_i(l) < \infty$.
    Therefore, any cycle with connected support satisfying $\dnorm{z} \leq l$ has a filling bounded by
    $D(l) = \max\{D_i(l) \mid 1 \leq i \leq m\} < \infty$.

    In the non-connected case one argues as in the proof of \cref{prop: finite fibers implies finite filling function}.
\end{proof}

\subsection{Weighted filling functions}
In the following let $\Gamma$ be a group of type $\FP_n(R)$.
We fix a free resolution $L_\ast \to R$ and finite $R[\Gamma]$\=/bases of $L_i$ for $i \leq n$.

\begin{lem}%
    \label{lem: bounded filling}
    Let $2 \leq i \leq n$.
    Assume that $R\noic$ is \(1\)\=/separated and $\weightedfilling{R,\Gamma}{i}(l) < \infty$ for all \(l\).
    For any $(i-1)$\=/cycle $c \in L_{i-1}$ there exists a reduced filling $y\noic$ such that
    \[\norm{y}^\Gamma \leq \norm{c}\norm{c}^\Gamma \! A_\Gamma^{i}(\norm{c}) \quad \text{ and } \quad \norm{y} \leq \norm{c}A_\Gamma^{i}(\norm{c})\mpct{,}\]
    where $K_{i-1}$ is the constant from \cref{lem: norm estimate connected} and
    \[A_\Gamma^i(x) = \weightedfilling{R,\Gamma}{i}(x + 2K_{i-1} x^2) + 1\mpct{.}\]
\end{lem}
\begin{proof}
    Because \(R\) is \(1\)\=/separated, \(c\) admits a decomposition $c = c_1 + \cdots + c_m$ into connected cycles with $m \leq \norm{c}$.
    For $h_j \in \supp_\Gamma(c_j)$ consider the cycles $c'_j \defq h_j^{-1}c_j$.
    Choose minimal fillings $y'_j$ of $c'_j$ such that
    \begin{equation*}
        \norm{y'_j}^\Gamma \leq \weightedfilling{R,\Gamma}{i}\bigl(\norm{c'_j}^\Gamma\bigr) + 1
    \end{equation*}
    and let $y_j \defq h_j y'_j$. The $+1$\=/term is needed as we cannot assume that the minimum in the definition of $\fillvolume[R]$ is attained.
    Then $\partial y_j = h_j \partial y'_j = h_j h_j^{-1}c_j = c_j$, hence
    $y \defq y_1 + \cdots + y_m$ is a filling of \(c\). We may assume that \(y\) is reduced, as otherwise we can remove the non-reduced part to obtain
    a smaller filling.
    Because $h_j \in \supp_\Gamma(c_j)$, there is some $v \in \supp_R(c_j)$ such that
    $v = h_j b_j$ for some $R\Gamma$\=/basis element $b_j$. Then $b_j \in \supp_R(c_j')$, hence by \cref{cor: estimate weighted to non-weighted}
    \begin{equation*}
        \norm{c'_j}^\Gamma
        \leq (\norm{b_j}^\Gamma + 2K_{i-1} \norm{c_j})\norm{c_j} = (1 + 2K_{i-1} \norm{c_j})\norm{c_j}
        \leq (1 + 2K_{i-1} \norm{c})\norm{c}.
    \end{equation*}
    It follows that
    \begin{align*}
        \norm{y'_j}^\Gamma
        \leq \weightedfilling{R,\Gamma}{i}\bigl(\norm{c'_j}^\Gamma\bigr) + 1
        \leq \weightedfilling{R,\Gamma}{i}\bigl((1 + 2K_{i-1} \norm{c})\norm{c}\bigr) + 1 = A_\Gamma^{i}(\norm{c}).
    \end{align*}
    Because \(R\) is \(1\)\=/separated, we have $1 + \ell_\Gamma(h_j) \leq \norm{c_j}^\Gamma \leq \norm{c}^\Gamma$. Hence,
    \begin{align*}
        \norm{y_j}^\Gamma
         &\leq (1 + \ell_\Gamma(h_j))\norm{y'_j}^\Gamma \leq \norm{c}^\Gamma A_\Gamma^{i}(\norm{c}).
    \end{align*}
    In total, we obtain that
    \begin{gather*}
        \norm{y}^\Gamma \leq \sum_{j=1}^m \norm{y_j}^\Gamma
        \leq \norm{c}\norm{c}^\Gamma A_\Gamma^{i}(\norm{c})\\
        \shortintertext{and}
        \norm{y} \leq \sum_{j=1}^m \norm{y_j} = \sum_{j=1}^m \norm{y'_j} \leq \sum_{j=1}^m \norm{y'_j}^\Gamma \leq \norm{c} A_\Gamma^i(\norm{c}).\qedhere
    \end{gather*}
\end{proof}

We now give a criterion for the finiteness of weighted filling functions in terms of the finiteness of
the non-weighted functions. The proof is based on the constructions in the proofs of~\cite{ji+ramsey}*{Corollary~2.5} and~\cite{bader+sauer}*{Lemma~6.7}.

\begin{prop}%
    \label{prop: polynomial equivalence}
    Assume that $R\noic$ is \(1\)\=/separated. For $2 \leq i \leq n$ there are constants $A_i, B_i > 0$ such that
    \begin{refenum}
        \item $\filling{R,\Gamma}{i}(x) \leq x \cdot \weightedfilling{R,\Gamma}{i}(x + A_i \cdot x^2) + 1$,
        \item $\weightedfilling{R,\Gamma}{i}(x) \leq B_i \cdot \bigl(\filling{R,\Gamma}{i}(x) + 1\bigr)^2 + x \cdot \bigl(\filling{R,\Gamma}{i}(x) + 1\bigr)$.
    \end{refenum}
    In particular $\filling{R,\Gamma}{i}$ is finite if and only if $\weightedfilling{R,\Gamma}{i}$ is finite.
\end{prop}
\begin{proof}
    Let $f(x) = \filling{R,\Gamma}{i}(x) + 1$, $c \in L_{i - 1}$ be an $(i - 1)$\=/cycle,
    and let $b = \sum_{u} b_u u \in L_i$
    be a reduced filling of c such that $\norm{b} \leq f(\norm{c})$.
    By \cref{lem: reduced filling estimate}, there exist $v_u \in \supp_R(c)$ such that
    \begin{align*}
        \norm{u}^\Gamma
        \leq \norm{v_u}^\Gamma + C_i \dnorm{b}
        \leq \norm{c}^\Gamma + C_i \norm{b},
    \end{align*}
    where in the second inequality we use that \(R\) is \(1\)-separated. Hence,
    \begin{align*}
        \norm{b}^\Gamma = \smash{\sum_{\mathclap{u \in \supp_R(b)}}} \abs{b_u}\norm{u}^\Gamma
         &\leq \norm{b}\bigr(\norm{c}^\Gamma + C_i\norm{b}\bigl) \\
         &\leq f(\norm{c})\norm{c}^\Gamma + C_i f(\norm{c})^2 \\
         &\leq f(\norm{c}^\Gamma)\norm{c}^\Gamma + C_i f(\norm{c}^\Gamma)^2\mpct{.}
    \end{align*}
    We conclude that $\weightedfilling{R,\Gamma}{i}(x) \leq C_i f(x)^2 + x \cdot f(x)$.

    Let $g(x) = \weightedfilling{R,\Gamma}{i}(x)$ and $c \in L_{i - 1}$ be a cycle. We first assume that $i \geq 2$ and
    consider the filling \(y\) of \(c\) from \cref{lem: bounded filling}. It satisfies
    $\norm{y} \leq \norm{c} g\bigl(\norm{c} + 2K_{i-1} \norm{c}^2\bigr) + 1$
    and thus,
    \[\filling{R,\Gamma}{i}(x) \leq x \cdot g(x + 2K_{i - 1}x^2) + 1.\qedhere\]
\end{proof}

\begin{rem}
    For $n = 1$, the function $\filling{R,\Gamma}{1}$ is finite if and only if $\Gamma$ is
    finite. This is because $\fillvolume[R](\gamma - 1)$ grows like $\ell_\Gamma(\gamma)$ for $\gamma \in \Gamma$.
    In contrast, $\weightedfilling{R,\Gamma}{1}$ is always linear, as $\norm{\gamma - 1}^\Gamma$ also grows linearly in
    $\ell_\Gamma(\gamma)$.
\end{rem}

\begin{cor}%
    \label{cor: discrete weighted finite}
    $\weighteddfilling{R,\Gamma}{n}$ takes finite values.
\end{cor}

\subsection{A non-finite example}%
\label{sec: finiteness:non-finite}
In~\cite{bader+kropholler+vankov25}, the authors ask the following question about the finiteness of filling functions.
\begin{question*}[\cite{bader+kropholler+vankov25}*{2.10}]%
    \label{question: finiteness}
    For which norms $\abs{\placeholder}$ does the \(n\)th homological filling function take finite values?
\end{question*}
We conclude this section with a short example, answering this question in the negative for all normed rings which are not $\epsilon$\=/separated.
This justifies why $\epsilon$\=/separation is a natural condition in the context of filling functions if one wants
$\filling{R,\Gamma}{n}$ to yield a finite invariant.

\begin{example}%
    \label{example: non-finite filling function}
    Assume that \(R\) is not $\epsilon$\=/separated for every $\epsilon > 0$.
    Then there exists a strictly increasing sequence $n_i \in \bbN$ and elements $r_i \in R$ such that $\frac{1}{n_i + 1} < \abs{r_i} \leq \frac{1}{n_i}$.
    By \cref{lem: symmetric norm,lem: equivalent filling from eqeuivalent norms}, we may assume that $\abs{r} = \abs{{-r}}$ for all $r \in R$.

    Let $\Gamma = \bbZ^2 = \langle x,y \mid [x,y]\rangle$ and \(X\) be the universal cover of its presentation complex, that is, $X = \bbR^2$ with
    $X^{(1)} = \bbZ^2$. Consider the \(1\)\=/cycles $\alpha_i = r_i \cdot [x^{n_i},y^{n_i}]$ in \(X\).
    They satisfy $\norm{\alpha_i} = \abs{r_i} \cdot 4n_i \leq 4$ and
    $\fillvolume[R](\alpha_i) = \abs{r_i} \cdot {n_i}^2 > \frac{n_i^2}{n_i + 1} > \frac{n_i}{2}$.
    In particular $\filling{R,\bbZ^2}{2}(4) \geq \frac{n_i}{2}$ for all \(i\), so $\filling{R,\bbZ^2}{2}$ is not finite.
\end{example}

However, under the assumption that \(R\) is $\epsilon$\=/separated one could expect a positive answer to the above question.

\section{Quasi-isometry invariance of filling functions}\label{sec: QI invariance}
The goal of this section is to prove quasi-isometric invariance of $\filling{R,\Gamma}{n}$ and $\weightedfilling{R,\Gamma}{n}$ for groups of type~$\FP_n(R)$
under the assumption that the homological filling functions of $\Gamma$ take finite values.
The proofs in this section build upon the proofs in~\cite{bader+kropholler+vankov25}*{Section 2}.

In the following, let \(G\) and \(H\) be finitely generated groups of type~$\FP_n(R)$. Let \(S\) and \(T\) be finite generating sets
for \(G\) and \(H\), respectively. Furthermore, let $L^G_\ast$ and $L^H_\ast$ be free resolutions as described in \cref{rem: cayley resolution}.

\begin{lem}%
    \label{lem: quasi-isometry partial chain map}
    Let \(R\) be \(1\)\=/separated.
    Assume that $\weightedfilling{R,H}{i}(l) < \infty$ for all $2 \leq i \leq n$ and all \(l\).
    Let $f \colon G \to H\noic$ be a $(K,K)$\=/quasi-isometry. Then for $i \leq n$ there exist \(R\)\=/linear maps $f_i \colon L^G_i \to L^H_i\Noic$ and constants $D_i$, $\tilde{D}_i$ such that
    \begin{refenum}
        \item $f_0$ is the \(R\)\=/linear extension of \(f\);
        \item $f_{i-1}\partial = \partial f_i$ for $1 \leq i \leq n$;
        \item $\norm{f_i(x)} \leq D_i \norm{x}$ for all $x \in L^G_i\Noic$ and $i \leq n$;
        \item $\norm{f_i(x)}^H \leq \tilde{D}_i \norm{x}^G$ for all $x \in L^G_i\Noic$ and $i \leq n$.
    \end{refenum}
\end{lem}
\begin{proof}
    Let $f_0 \colon RG \to RH$ be the linear extension of \(f\). Then $\norm{f_0(g)} = \norm{g}$ for $g \in G$ and
    \[
        \norm{f_0(g)}^H = 1 + \ell_H(f(g)) \leq 1 + K\ell_G(g) + K \leq (1+K) \norm{g}^G\mpct{.}
    \]
    We set $D_0 \defq 1$ and $\tilde{D}_0 \defq 2K$.

    An \(R\)\=/basis of $L^G_1 = RG^{\size{S}}$ is given by $\{g b_s \mid g \in G,\, s \in S\}$. For such a basis element $g b_s$ we
    have $\partial g b_s = g (s - 1) = g s - g$.
    Let $\gamma_{g,s}$ be a path from $f_0(g)$ to $f_0(gs)$ in $\Cay(H, T)$.
    Since \(f\) is a $(K,K)$\=/quasi-isometry, we can arrange that $\gamma_{g,s}$ has length at most $2K + 1$.
    Because we have assumed that the resolution $L^H_\ast$ coincides with $C_i(\Cay(H,S);R)$ for $i=0,1$, we may
    consider $\gamma_{g,s}$ as an element of $L^H_1 = RH^{\size{T}}$ and define $f_1(gb_s) \defq \gamma_{g,s}$.
    Then $\gamma_{g,s}$ satisfies $\norm{\gamma_{g,s}} \leq (2K + 1) \abs{1_R}$ and we set $D_1 \defq 2K + 1$.
    Write $\gamma_{g,s} = \sum_{j = 0}^m e_j$ with $m \leq 2K$.
    It follows from \cref{lem: norm estimate connected} that $\norm{e_j}^H \leq \norm{e_0}^H + 2K \cdot 2K_1$ and $\norm{e_0}^H \leq \norm{f_0(g)}^H + K_1$.
    Therefore,
    \begin{align*}
        \norm{\gamma_{g,s}}^H
         &\leq (2K + 1)(\norm{e_0}^H + 2K \cdot 2K_1) \\
         &\leq (2K + 1)(\tilde{D}_0\norm{gb_s}^G + K_1 + 4KK_1) \\
         &\leq (2K + 1)(\tilde{D}_0 + K_1 + 4KK_1)\norm{gb_s}^G\mpct{,}
    \end{align*}
    where we use that $\norm{f_0(g)}^H \leq \tilde{D}_0 \norm{g}^G = \tilde{D}_0\norm{gb_s}^G$ in the second inequality.
    We define $\tilde{D}_1 \defq (2K + 1)(\tilde{D}_0 + K_1 + 4KK_1)$.

    Now assume that $f_i$ has been defined for $i < k \leq n$. Let $gb$ be some \(R\)\=/basis element of $L^G_k$. Then
    \[\partial f_{k-1}(\partial gb) = f_{k-2}(\partial \partial gb) = f_{k-2}(0) = 0\mpct{,}\]
    so $c \defq f_{k-1}(\partial g b)$ is a cycle. We have
    \begin{align*}
         &\norm{c} = \norm{f_{k-1}(\partial gb)} \leq D_{k-1}\norm{\partial gb} = D_{k-1}\norm{\partial b} \leq D_{k-1}C_k\mpct{,} \\
         &\norm{c}^H = \norm{f_{k-1}(\partial gb)}^H \leq \tilde{D}_{k-1} \norm{\partial g b}^G \leq M_k \tilde{D}_{k-1} \norm{gb}^G\mpct{,}
    \end{align*}
    where $C_k \defq \max\{ \norm{\partial b} \mid b \text{ an } RG\text{-basis element of } L^G_k \}$ and
    $M_k$ is the constant obtained from \cref{lem: equivariant bounded} applied to $\partial_k \colon L^G_k \to L^G_{k-1}$.
    \pagebreak
    Let \(y\) be the filling of \(c\) obtained from \cref{lem: bounded filling} and define $f_k(gb) \defq y$. Then
    \begin{align*}
        \norm{y}^H
         &\leq \norm{c}\norm{c}^H A_H^k(\norm{c}) \\
         &\leq \norm{c}^H D_{k-1}C_k \cdot A_H^k(D_{k-1}C_k) \\
         &\leq \norm{gb}^G \underbrace{M_k \tilde{D}_{k-1} D_{k-1}C_k \cdot A_H^k(D_{k-1}C_k)}_{\qdef \tilde{D}_k}
    \end{align*}
    and
    \[
        \norm{y} \leq \norm{c} A_H^k(\norm{c}) \leq D_{k-1}C_k \cdot A_H^k(D_{k-1}C_k) \qdef D_k\mpct{.}\qedhere
    \]
\end{proof}

\begin{lem}%
    \label{lem: quasi-isometry partial homotopy}
    Let $R\noic$ be \(1\)\=/separated.
    We retain the assumptions of \cref{lem: quasi-isometry partial chain map} up to degree $n-1$.
    Let $f \colon G \to H\noic$ and $h \colon H \to G\noic$ be $(K,K)$\=/quasi-isometries such that $h \circ f$ is \(K\)\=/close to
    $\id_G$. Let $f_\ast$ and $h_\ast$ be the partial chain maps obtained from \cref{lem: quasi-isometry partial chain map}.
    Further assume $\filling{R,G}{n}(l) < \infty$ for all \(l\).

    Then for $i \leq n - 1$ there exist \(R\)\=/linear maps $s_i \colon L^G_i \to L^G_{i+1}$ and constants $E_i$, $\tilde{E}_i$ such that
    \begin{refenum}
        \item $\partial s_i + s_{i-1} \partial = \id - \,h_i \circ f_i$ for $i \leq n - 1$;
        \item $\norm{s_i(x)} \leq E_i \norm{x}$ for all $x \in L^G_i$ and $i \leq n - 1$;
        \item $\norm{s_i(x)}^G \leq \tilde{E}_i \norm{x}^G$ for all $x \in L^G_i$ and $i \leq n - 1$.
    \end{refenum}
\end{lem}
\begin{proof}
    For $g \in G$ let $\gamma_g$ be a path from \(g\) to $h(f(g))$ in $\Cay(G, S)$. Since $h \circ f$ is \(K\)\=/close to $\id_G$,
    we may choose $\gamma_g$ such that it has length at most \(K\). We define $s_0$ to be the \(R\)\=/linear extension such that
    $s_0(g) = \gamma_g$. Then $\norm{s_0(g)} \leq K \norm{g}$ and we set $E_0 \defq K$.
    Similar to the beginning of the proof of \cref{lem: quasi-isometry partial chain map} we compute that
    $\norm{\gamma_g}^G \leq \tilde{E}_1 \norm{g}^G$, where $\tilde{E}_1 \defq K (1 + (K - 1) \cdot 2K_1)$.

    Assume that we have constructed $s_{i-1} \colon L^G_{i-1} \to L^G_i$ for $i \leq n - 1$ and let $c = g b$ be an \(R\)\=/basis element of $L^G_i$.
    Note that
    \[
        \partial s_{i-1}(\partial c) = \partial c  - \partial h_i(f_i(c)) = \partial \bigl(c - h_i(f_i(c))\bigr)\mpct{,}
    \]
    hence $c' \defq c - h_i(f_i(c)) - s_{i-1}(\partial c)$ is a cycle.
    As in the proof of \cref{lem: quasi-isometry partial chain map} we set $C_i = \max\{ \norm{\partial b} \mid b \text{ an } RG\text{-basis element of } L^G_i \}$
    to conclude that $\norm{c'} \leq \abs{1_R} + D'_i D_i\abs{1_R} + E_{i-1}C_i \qdef F_i$.
    Define $s_i(c)$ to be the filling \(y\) of $c'$ obtained from \cref{lem: bounded filling}.
    Then
    \begin{align*}
        \norm{y}
         &\leq \norm{c'} A_G^{i+1}(\norm{c'}) \leq F_i A_G^{i+1}(F_i) \qdef E_i\mpct{,} \\
        \norm{y}^G
         &\leq \norm{c'}\norm{c'}^G A_G^{i+1}(\norm{c'}) \leq \underbrace{E_i \cdot (1 + \tilde{D}_i\tilde{D}'_i + \tilde{E}_{i-1}M_i)}_{\qdef \tilde{E}_i} \norm{c}^G\mpct{,}
    \end{align*}
    where $M_i$ is the constant obtained from \cref{lem: equivariant bounded} applied to $\partial_i$.
\end{proof}

\begin{rem}
    The assumption that \(R\) is \(1\)\=/separated is only needed for the bound on the weighted norm in \cref{lem: bounded filling}.
    If one is only interested in the non-weighted filling functions one can always construct the respective maps from
    \cref{lem: quasi-isometry partial chain map,lem: quasi-isometry partial homotopy} only satisfying bounds in the non-weighted norm.
    Without the need for a bound on the weighted norm \cref{lem: bounded filling} can be replaced by the observation that
    every cycle \(c\) admits a filling \(y\) such that $\norm{y} \leq \filling{R,\Gamma}{i}(\norm{c}) + 1$.
    For details see the proofs of~\cite{bader+kropholler+vankov25}*{Section 2}.
    If one works with free resolutions as in \cref{rem: cayley resolution} the proofs can be followed almost verbatim.
\end{rem}

\begin{thm}[store=qiInvarianceIfFinite]%
    \label{thm: quasi-isometry invariance of filling}
    Let $G\noic$ and $H\noic$ be quasi-isometric groups of type $\FP_n(R)$.
    Assume $\filling{R,G}{i}(l),\filling{R,H}{i}(l)<\infty$ for all $2 \leq i\leq n$ and all \(l\).
    Then $\filling{R,G}{n}\approx\filling{R,H}{n}$.
\end{thm}
\begin{proof}
    Let $f \colon G \to H$ be a quasi-isometry with quasi-inverse $h \colon H \to G$, $f_\ast, h_\ast$ be the partial chain maps obtained from \cref{lem: quasi-isometry partial chain map} and $D_i$, $D'_i$ the associated constants.

    Let $s_i \colon L^G_i \to L^G_{i + 1}$ and $E_i$ be the maps and constants obtained from \cref{lem: quasi-isometry partial homotopy}.
    If $z \in L^G_{n-1}$ is an $(n-1)$\=/cycle, so is $f_{n-1}(z)$ and $\norm{f_{n-1}(z)} \leq D_{n-1} \norm{z}$.
    By assumption, there exists a filling $b \in L^H_{n}$ of $f_{n-1}(z)$ such that $\norm{b} \leq \filling{R,H}{n}(D_{n-1}\norm{z}) + 1$. Then $h_n(b) + s_{n-1}(z)$ is a filling of \(z\), as
    \begin{align*}
        \partial h_n(b) + \partial s_{n-1}(z)
         &= h_{n-1}(\partial b) + \partial s_{n-1}(z) \\
         &= h_{n-1}(f_{n-1}(z)) + \bigl(z - h_{n-1}f_{n-1}(z) - s_{n-2}(\partial z)\bigr) = z\mpct{.}
    \end{align*}
    Moreover,
    \begin{align*}
        \norm{h_n(b) + s_{n-1}(z)}
         &\leq D'_n \norm{b} + E_{n-1}\norm{z} \\
         &\leq D'_n \left(\filling{R,H}{n}(D_{n-1}\norm{z}) + 1\right) + E_{n-1}\norm{z}\mpct{,}
    \end{align*}
    so $\filling{R,G}{n} \preccurlyeq \filling{R,H}{n}$. The argument is symmetric, completing the proof.
\end{proof}

As a direct consequence of \cref{thm: quasi-isometry invariance of filling,thm: dfilling finite} we obtain that $\dfilling{R,G}{n}$ is a quasi-isometry invariant for groups of type~$\FP_n(R)$.

\begin{thm}[store=qiInvarianceDiscrete]%
    \label{thm: QI-invariance discrete}
    Let $G\noic$ and $H\noic$ be quasi-isometric groups of type $\FP_n(R)$.
    Then $\dfilling{R,G}{n}\approx\dfilling{R,H}{n}$.
\end{thm}

\begin{rem}
    It was proven by Alonso~\cite{alonso} that $\FP_n(R)$ is a quasi-isometry invariant.
    We could therefore weaken the hypothesis of \cref{thm: QI-invariance discrete,thm: quasi-isometry invariance of filling}
    to just requiring \(G\) to be of type $\FP_n(R)$.
    This is not possible for groups of type $\FH_n(R)$, where it is still an open problem, whether
    $\FH_n(R)$ is invariant under quasi-isometry.
    Note that for $n = 2$ \cref{thm: dfilling finite} is a consequence of~\cite{bader+kropholler+vankov25}*{2.21},
    as $\FP_n(R)$ and $\FH_n(R)$ coincide for $n \leq 2$.
\end{rem}

In~\cite{fleming+martinez-pedroza}, Joshua W.\ Fleming and Eduardo Martínez-Pedroza proved that for $R = \bbZ$ equipped with the absolute value
$\filling{\bbZ,G}{n}$ takes finite values (\cref{cor: integral finite} gives another proof of this fact).
Therefore, the integral homological Dehn functions of \(G\)
are quasi-isometry invariants for groups of type~$\FP_n(\bbZ)$.
This has probably already been known to experts in the field, but while there may be a reference for this result, the author is not aware of any.
We therefore record it here for the reader's convenience.

\begin{thm}[\cites{fletcher,young,fleming+martinez-pedroza}]%
    \label{cor: QI-invariance integral}
    Let $G\noic$ and $H\noic$ be quasi-isometric groups of type $\FP_n(\bbZ)$.
    Then $\filling{\bbZ,G}{n} \approx \filling{\bbZ,H}{n}$.
\end{thm}

Using \cref{prop: polynomial equivalence} we obtain analogous results for the weighted versions of these filling functions.

\begin{thm}%
    \label{thm: quasi-isometry invariance of weighted filling}
    Let $R\noic$ be $\epsilon$\=/separated.
    Assume that $\weightedfilling{R,G}{i}(l), \weightedfilling{R,H}{i}(l) < \infty$ for $i \leq n$ and all \(l\).
    If $G\noic$ and $H\noic$ are quasi-isometric, then $\weightedfilling{R,G}{n} \approx \weightedfilling{R,H}{n}$.
\end{thm}
\begin{proof}
    Due to \cref{lem: 0-separated norm bounded by 1,lem: equivalent filling from eqeuivalent norms} we may assume that
    \(R\) is \(1\)\=/separated. Now the proof of \cref{thm: quasi-isometry invariance of filling} applies verbatim with
    $D_i$, $D'_i$, $E_i$ replaced by $\tilde{D}_i$, $\tilde{D}'_i$, $\tilde{E}_i$ and $\norm{\placeholder}$ replaced by
    $\norm{\placeholder}^G$.
\end{proof}

\begin{thm}[store=qiInvarianceDiscreteWeighted]%
    \label{thm: QI-invariance discrete weighted}
    Let $G\noic$ and $H\noic$ be quasi-isometric groups of type $\FP_n(R)$.
    Then $\weighteddfilling{R,G}{n}\approx\weighteddfilling{R,H}{n}$.
\end{thm}

\begin{thm}[store=qiInvarianceIntegralWeighted]%
    \label{thm: QI-invariance integral weighted}
    Let $G\noic$ and $H\noic$ be quasi-isometric groups of type $\FP_n(\bbZ)$.
    Then $\weightedfilling{\bbZ,G}{n}\approx\weightedfilling{\bbZ,H}{n}$.
\end{thm}

Even though there is a close relationship between weighted and unweighted filling functions due to \cref{prop: polynomial equivalence},
it is unclear whether knowledge of one of the functions allows us to recover the other.

\begin{question}
    Are there groups \(G\), $H\noic$ of type $\FP_n$ such that $\filling{\bbZ,G}{n} \approx \filling{\bbZ,H}{n}$ but
    $\weightedfilling{\bbZ,G}{n} \not\approx \weightedfilling{\bbZ,H}{n}$ or vice versa?
\end{question}

\section{Cohomology}%
\label{sec: cohomological dimension}
This section provides an algebraic proof of the quasi-isometry invariance of group cohomology with
coefficients in the group ring, originally established by Gersten~\cite{gersten} for finitely presented
groups and more generally by Li~\cite{lixin} for groups of type $\FP(R)$.
The argument presented here follows Gersten's proof in spirit without relying on the groups being finitely presented.
As a consequence, we deduce quasi-isometry invariance of the cohomological dimension $\cd_R(G)$ and of being a (Poincaré) duality group for groups of type $\FP(R)$ -- results originally due to Sauer~\cite{sauer}, respectively Li~\cite{lixin}.
The proofs use the techniques developed in \cref{sec: cellular arguments} and \cref{sec: QI invariance}.

In the following let \(G\) and \(H\) be groups of type $\FP(R)$ for a ring \(R\).
Furthermore, let $L^G_\ast$ and $L^H_\ast$ be free resolutions of \(R\) as an $RG$- respectively $RH$-module
such that all $L^G_i$ and $L^H_i$ are finitely generated.
Let $B_i$ be choices of $RG$-bases of $L^G_i$.

For this section we consider \(R\) to be equipped with the discrete norm. In this case \cref{thm: dfilling finite}
asserts that $\dfilling{R,G}{i}$ and $\dfilling{R,H}{i}$ are finite functions for all $i \geq 2$.

\begin{lem}%
    \label{lem: bound basis to chain upgrade}
    Let $f \colon G \to H$ be a map and $\varphi \colon RG^n \to RH^m$ be an \(R\)-linear map such that
    there exists a constant $C > 0$ with
    $\supp_H(\varphi(gb_i)) \subseteq B_H(f(g), C)$ for all $g \in G$ and $1 \leq i \leq n$.
    Then $\supp_H(\varphi(c)) \subseteq B_H(f(\supp_G(c)), C)$ for any $c \in RG^n$.
\end{lem}
\begin{proof}
    Apply the assumption to each basis element in $\supp_R(c)$ to obtain
    \[
        \supp_H(\varphi(c))
        \subseteq \;\bigcup_{\mathclap{gb \in \supp_R(c)}}\; \supp_H(\varphi(gb))
        \subseteq \;\bigcup_{\mathclap{gb \in \supp_R(c)}}\; B_H\bigl(f(g),C\bigr)
        = B_H\bigl(f(\supp_G(c)),C\bigr)\mpct{.}\qedhere
    \]
\end{proof}

\begin{defn}
    Let $\varphi \colon L \to L'$ be an \(R\)-linear map between free \(R\)-modules with bases \(B\) and $B'$ respectively.
    We say that $\varphi$ is \emph{algebraically proper} if for each $y \in B'$ the set
    \[
        I_\varphi(y) \defq \{x \in B \mid y \in \supp_R(\varphi(x)) \}
    \]
    is finite.
\end{defn}

\begin{lem}%
    \label{lem: proper from bounded G-support}
    We retain the assumptions of \cref{lem: bound basis to chain upgrade}.
    If \(f\) is coarsely injective then $\varphi$ is algebraically proper.
\end{lem}
\begin{proof}
    Let \(B\) be an $RG$-basis for $RG^n$ and $B'$ an $RH$-basis for $RH^m$,
    such that $GB$ and $HB'$ are the corresponding \(R\)-bases, respectively.
    Consider $y = hb' \in HB'$ and $x = gb \in I_{\varphi}(y)$. Then $h \in B_H(f(g), C)$, hence $f(g) \in B_H(h, C)$.
    Therefore, $I_\varphi(y) \subseteq f^{-1}(B_H(h, C)) \cdot B$,
    which is finite as \(f\) is coarsely injective and \(B\) is finite.
\end{proof}

\begin{prop}%
    \label{prop: group support contained in ball}
    Let $f \colon G \to H$ be a $(K,K)$\=/quasi-isometry and
    $f_i \colon L^G_i \to L^H_i$ be the \(R\)-linear chain map constructed
    in \cref{lem: quasi-isometry partial chain map}.

    Then there are constants $C_i \geq 0$ such that
    $\supp_H(f_i(x)) \subseteq B_H\bigl(f(\supp_G(x)), C_i\bigr)$
    for all $x \in L^G_i$ and all $i \geq 0$.
\end{prop}
\begin{proof}
    In view of \cref{lem: bound basis to chain upgrade} it suffices to consider the case $c = gb \in GB_i$.
    For $i = 0$ we have $\supp_H(f_0(g)) = \{f(g)\}$, so we can take $C_0 = 0$.
    Assume that the statement holds for $i - 1 \geq 0$. Then
    \begin{equation}%
        \label{eq: bound induction hypothesis}
        \begin{split}
            \supp_H(f_{i-1}(\partial gb))
             & \subseteq B_H\bigl(f(\supp_G(g \partial b)), C_{i-1}\bigr) \\
             & = B_H\bigl(f(g\supp_G(\partial b)), C_{i-1}\bigr)\mpct{.}
        \end{split}
    \end{equation}
    for any $gb \in GB_i$.
    Because \(f\) is a $(K,K)$-quasi-isometry, we have
    \[d_H(f(g), f(gg')) \leq K d_G(g, gg') + K = K\ell_G(g') + K\]
    for every $g' \in \supp_G(\partial b)$.
    Set $C_b \defq \max\{K \ell_G(g') + K \mid g' \in \supp_G(\partial b)\}$.
    Then
    \begin{align}
        f(g \supp_G(\partial b))      &\subseteq B_H(f(g), C_b)\nonumber \\
        \intertext{%
            and thus, using \cref{eq: bound induction hypothesis}, every $gb \in GB_i$ satisfies%
        }
        \supp_H(f_{i-1}(\partial gb)) &\subseteq B_H(f(g), C_{i-1} + \tilde{C}_i)\mpct{,}\label{eq: bound H-support}
    \end{align}
    where $\tilde{C}_i \defq \max\{C_b \mid b \in B_i\}$.

    \Cref{lem: reduced filling estimate} yields a constant $C_i'$ such that every $u \in \supp_R(f_{i}(gb))$ satisfies
    $d_H(u, v_u) \leq C_i' \dnorm{f_i(gb)}$
    for some $v_u \in \supp_R(f_{i-1}(\partial gb))$.
    Here we use that $f_{i}(gb)$ is reduced.
    Writing $u = h_u b_u$, and $v_u = h_v b_v$ for $h_u,h_v\in H$, $b_u \in B_i$, and $h_v \in B_{i-1}$
    we see that $h_v \in \supp_H(f_i(gb))$ and
    \begin{equation}%
        \label{eq: bound word length cell in filling}
        d_H(h_u,h_v) \leq C_i' D_i\dnorm{gb} = C_i' D_i\mpct{,}
    \end{equation}
    where $D_i$ is the constant from \cref{lem: quasi-isometry partial chain map}.
    Together, \cref{eq: bound H-support,eq: bound word length cell in filling} yield
    \[
        \supp_H(f_i(gb)) \subseteq B_H(f(g), C_{i-1} + \tilde{C}_i + C_i'D_i)\mpct{.}\qedhere
    \]
\end{proof}

\begin{prop}%
    \label{prop: group support homotopy contained in ball}
    Let $f \colon G \to H$ be a $(K,K)$\=/quasi-isometry with quasi-inverse $h \colon H \to G$
    such that $h \circ f$ is \(K\)-close to $\id_G$.
    Let $f_i \colon L^G_i \to L^H_i$, $h_i \colon L^H_i \to L^G_i$ be the \(R\)-linear chain maps
    constructed in \cref{lem: quasi-isometry partial chain map}.
    Further, let $s_i \colon L^G_i \to L^G_{i+1}$ be the \(R\)-linear homotopy from
    \cref{lem: quasi-isometry partial homotopy}.

    Then there are constants $A_i \geq 0$ such that for all $x \in L^G_i$
    \[
        \supp_G(s_i(x)) \subseteq B_G(\supp_G(x), A_i)\mpct{.}
    \]
\end{prop}
\begin{proof}
    As in the proof of \cref{prop: group support contained in ball} we can restrict to the case
    $c = gb \in GB_i$.
    For $i=0$ an element $g \in G = GB_0$ gets mapped to a geodesic path $s_0(g)$ from \(g\) to $h(f(g))$ of length at most \(K\).
    \Cref{cor: estimate weighted to non-weighted} then tells us that $\supp_G(s_0(g))$ is contained in $B_G(g, 2K_1K)$.

    Now assume that the statement holds for $i - 1 \geq 0$.
    There exist $C_i > 0$ and $C'_i > 0$ such that
    \begin{equation}%
        \label{eq: support bound fh}
        \begin{split}
            \supp_G\bigl(h_i(f_i(c))\bigr)
             & \subseteq B_G\Bigl(h\bigl(\supp_H(f_i(c))\bigr), C'_i\Bigr)                 \\
             & \subseteq B_G\Bigl(h\bigl(B_H(f(\supp_G(c)), C_i)\bigr), C'_i\Bigr)\mpct{.}
        \end{split}
    \end{equation}
    for any $c \in L^G_{i}$.
    Let $g \in \supp_G(c)$ and $g' \in B_H(f(g), C_i)$. Then
    \begin{align*}
        d_G(h(g'), g)
         &\leq d_G\bigl(h(g'), h(f(g))\bigr) + d_G\bigl(h(f(g)), g\bigr) \\
         &\leq \bigl(K \cdot d_H(g',f(g)) + K\bigr) + K \\
         &\leq KC_i + 2K\mpct{,}
    \end{align*}
    hence $h\bigl(B_H(f(g), C_i)\bigr) \subseteq B_G(g, KC_i + 2K)$.
    Now \cref{eq: support bound fh} yields
    \begin{align*}
        \supp_G\bigl(h_i(f_i(c))\bigr)
         &\subseteq B_G\Bigl(h\bigl(B_H(f(\supp_G(c)), C_i)\bigr), C'_i\Bigr) \\
         &\subseteq\;\bigcup_{\mathclap{g \in \supp_G(c)}}\; B_G\Bigl(B_G(g, KC_i + 2K), C'_i\Bigr) \\
         &= B_G\bigl(\supp_G(c), KC_i + 2K + C'_i\bigr)
    \end{align*}
    for every $c \in L^G_i$. For $c = gb \in GB_i$ the cycle $c' = c - h_i(f_i(c)) - s_{i-1}(\partial c)$ satisfies
    \begin{align*}
        \supp_G(c')
         &\subseteq \{g\} \cup B_G(g, KC_i + 2K + C'_i) \cup B_G(g \supp(\partial b), A_{i-1}) \\
         &\subseteq B_G\bigl(g, \max(KC_i + 2K + C'_i, \tilde{A}_i + A_{i-1})\bigr)\mpct{,}
    \end{align*}
    where $\tilde{A}_i \defq \max\{\ell_G(g') \mid g' \in \bigcup_{b \in B_i} \supp_G(\partial b_i)\}$.
    By definition, $s_i(c)$ is a reduced minimal filling of $c'$ and one proceeds as in the last step of the proof of \cref{prop: group support contained in ball}
    to obtain a constant $A_i$ such that $\supp_G(s_i(c)) \subseteq B_G(g, A_i)$.
\end{proof}

\Cref{prop: group support contained in ball,prop: group support homotopy contained in ball} together with \cref{lem: bound basis to chain upgrade} imply that the maps constructed in \cref{lem: quasi-isometry partial chain map,lem: quasi-isometry partial homotopy} are algebraically proper, hence we obtain the following \lcnamecref{cor: proper chain equivalence from QI}.

\begin{cor}%
    \label{cor: proper chain equivalence from QI}
    Let $f \colon G \to H$ be a quasi-isometry.
    There exists an \(R\)-linear chain homotopy equivalence $f_\ast \colon L^G_\ast \to L^H_\ast$ with
    chain homotopy inverse $h_\ast \colon L^H_\ast \to L^G_\ast$ and homotopies $s_\ast \colon h_\ast \circ f_\ast \simeq \id$,
    $t_\ast \colon f_\ast \circ h_\ast \simeq \id$ such that $f_i$, $h_i$, $s_i$ and $t_i$ are algebraically proper
    for all \(i\).
\end{cor}

For a group $\Gamma$ and a left $R\Gamma$\=/module \(M\) consider the right $R\Gamma$\=/module
$\fhom_R(M, R)$ consisting of those \(R\)\=/linear maps $f \colon M \to R$
such that, for every $m \in M$, we have $f(\gamma m) = 0$ for all but finitely many $\gamma \in \Gamma$.
By~\cite{brown}*{VIII Lemma 7.4}, there is a natural isomorphism
\begin{equation}%
    \label{eq: iso fhom}
    \hom_{R\Gamma}(M,R\Gamma) \cong \fhom_R(M,R)\mpct{.}
\end{equation}
If \(M\) is a free $R\Gamma$-module with $R\Gamma$-basis \(B\), the module $\fhom_R(M,R)$ consists of those maps \(f\)
for which $f(\gamma b) \neq 0$ for only finitely many $\gamma b \in \Gamma B$.

Evidently, an algebraically proper map $\varphi \colon RG^n \to RH^m$ induces a well-defined map
$\fhom_R(RH^m,R) \to \fhom_R(RG^n,R)$ through precomposition.

\begin{thm}[\cite{lixin}*{Corollary 4.41}]%
    \label{thm: QI-invariance cohomology}
    Let \(G\) and \(H\) be groups of type $\FP(R)$.
    If \(G\) and \(H\) are quasi-isometric then $H^i(G; RG) \cong H^i(H; RH)$ as \(R\)-modules for all $i \geq 0$
    and any ring \(R\).
\end{thm}
\begin{proof}
    The chain homotopy equivalence from \cref{cor: proper chain equivalence from QI} induces a well-defined
    \(R\)-linear chain homotopy equivalence $\fhom_R(L^H_\ast,R) \to \fhom_R(L^G_\ast,R)$.
    Thus,
    \begin{align*}
        H^i(G;RG)
         &\cong H^i(\hom_{RG}(L^G_\ast, RG)) \cong H^i(\fhom_{R}(L^G_\ast, R)) \\
         &\cong H^i(\fhom_{R}(L^H_\ast, R)) \cong H^i(\hom_{RH}(L^H_\ast, RH))
        \cong H^i(H;RH)\mpct{,}
    \end{align*}
    where the second and third isomorphisms use \cref{eq: iso fhom}.
\end{proof}

\begin{cor}[\cite{sauer}*{Theorem 1.2}]%
    \label{cor: QI-invariance cd}
    The cohomological dimension $\cd_R(G)$ is a quasi-isometry invariant of groups of type $\FP(R)$.
\end{cor}

Another consequence of \cref{thm: QI-invariance cohomology} is the quasi-isometry invariance of being a (Poincar\'e) duality group.
The proof is the same as the one of Corollary 3 in~\cite{gersten}.

\begin{cor}[\citelist{\cite{gersten}*{Corollary 3}\cite{lixin}*{Corollary 4.41}}]%
    \label{cor: QI-invariance duality}
    Let \(G\) and \(H\) be quasi-isometric groups of type $\FP(R)$. If \(G\) is a duality group, respectively a Poincar\'e duality group over \(R\), then so is \(H\).
\end{cor}

\bibliography{references}

\end{document}